\documentclass[a4paper,12pt]{amsart}
\usepackage[utf8]{inputenc}
\usepackage[margin=1.3in,headsep=.23in]{geometry}
\usepackage{graphicx}
\usepackage{amssymb,amsmath,amsthm}
\usepackage{tikz-cd}
\usetikzlibrary{decorations.markings}
\usepackage{xcolor}
\usepackage{url}
\usepackage[colorlinks=true,linkcolor=red,citecolor=blue]{hyperref}
\usepackage{cleveref}
\usepackage[all,cmtip]{xy}
\usepackage[foot]{amsaddr}
\usepackage[numbers,sort&compress]{natbib}
\usepackage[notcite,notref]{}
\usepackage{comment}

\newtheorem{thm}{Theorem}
\newtheorem{prop}[thm]{Proposition}
\newtheorem{lem}[thm]{Lemma}
\newtheorem{cor}[thm]{Corollary}
\theoremstyle{definition}
\newtheorem{defn}[thm]{Definition}
\newtheorem{example}[thm]{Example}
\AtBeginEnvironment{example}{%
  \pushQED{\qed}%
}
\AtEndEnvironment{example}{\popQED\endexample}
\newtheorem{rem}[thm]{Remark}
\AtBeginEnvironment{rem}{%
  \pushQED{\qed}%
}
\AtEndEnvironment{rem}{\popQED\endexample}
\numberwithin{equation}{section}
\numberwithin{thm}{section}

\newcommand{\orcidicon}[1]{\href{https://orcid.org/#1}{\includegraphics[height=2.5ex]{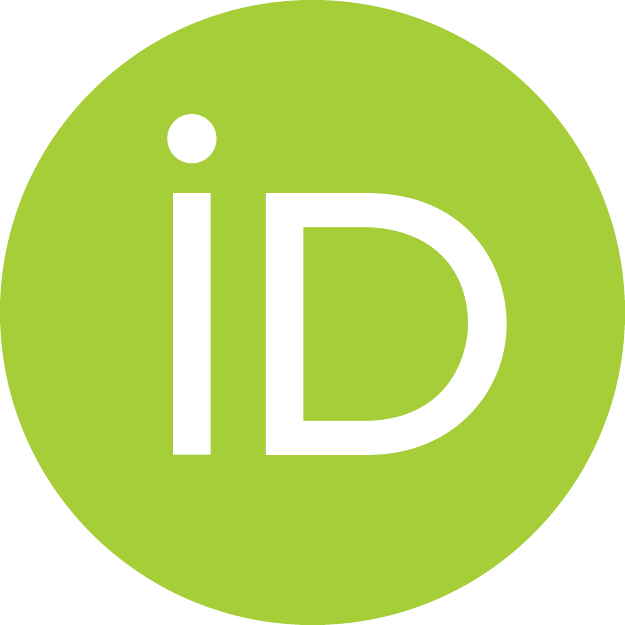}}}

\newcommand{\inv}[1]{ {#1}^{-1} }
\newcommand{\m}[1]{\mathcal{#1}}
\newcommand{\mb}[1]{\mathbb {#1}}
\newcommand{\RR}{\mathbb{R}}

\newcommand{\ZZ}{\mathbb{Z}}
\newcommand{\NN}{\mathbb{N}}
\newcommand{\QQ}{\mathbb{Q}}
\newcommand{\PP}{\mathbb{P}}
\newcommand{\be}{\begin{equation}} 
\newcommand{\ee}{\end{equation}} 
\newcommand{\ti}[1]{\widetilde {#1}}

\newcommand{\image}[1]{\mathrm{im}\ \! {#1}}

\DeclareMathOperator{\Aut}{Aut}

\DeclareMathOperator{\id}{id}

\newcommand{\set}[1]{\{#1\}}

\newcommand{\ori}{\eta}
\newcommand{\comm}{\mathsf{Com}}
\newcommand{\ass}{\mathsf{Ass}}
\newcommand{\lie}{\mathsf{Lie}}
\newcommand{\gf}{\mathsf{GF}}
\newcommand{\gp}{\mathsf{GP}}
\newcommand{\homgra}{H_1(G)}
\newcommand{\tad}{ \!\!\! {\begin{tikzpicture}[baseline=-1,scale=.66, every loop/.style={}] \coordinate  (v) at (0,0);
    \filldraw[fill=black] (v) circle (0.035);
    \draw (v) edge[loop] (v);
\end{tikzpicture}}}
\newcommand{\cube}{c(G,\gamma)}
\newcommand{\sla}{/ \! /}
\newcommand{\val}[1]{\mathrm{val}({#1})}

\begin{document}

\title{Graph complexes from the geometric viewpoint}

\author{Marko Berghoff \protect\orcidicon{0000-0002-9108-3045}}
\email{berghoffmj@gmail.com}
\address{Institut für Mathematik, Humboldt-Universit\"at zu Berlin, Germany}

\begin{abstract}
%These notes loosely follow an introductory course on graph complexes, held at Humboldt-Universität zu Berlin in summer 23. Instead of simply typing up my lecture notes I decided to give here an overview over (parts of) the topic (lecture notes can be found on my homepage).

We introduce the associative, commutative and Lie graph complexes, and moduli spaces of metric graphs, then discuss how the commutative and Lie graph complexes can be interpreted as cellular chain complexes associated to certain (pairs of) subspaces of the latter, both for ``even" and ``odd" orientations. We explain why this does not work for the associative complex and how to adjust the space of graphs to deal with this case.
Along the way we highlight how algebraic properties on one side translate into geometric statements on the other. 
\end{abstract}

\maketitle

\section{Introduction}

\subsection{Overview}
In \cite{ko1,ko2} Kontsevich introduced three types of chain complexes generated by finite graphs, with a differential defined by collapsing edges. He showed how the homology of these complexes relates to various invariants in low-dimensional topology and geometric group theory. Despite the simple combinatorial definition of these \textit{graph complexes}, computing and understanding their homology is a challenging open problem.
By now there exist many accounts on the topic, approaching it from quite different angles, many of which are very algebraic in nature. 
Although they lead to powerful methods, these algebraic approaches are often hard to digest, or even visualise. This is especially true for absolute beginners\footnote{``\textit{Fuchs' mich in die Materie, da es Möglichkeiten unbegrenzt gibt}", Samy Deluxe in \cite{beginner}.}.

Some ideas on how to view and study graph complexes from a more geometric point of view are already contained in Kontsevich's orginal works. This line of thought has been continued in many papers, most prominently in \cite{convogt} which is entirely devoted to untangle and clarify many of the arguments and constructions in \cite{ko1,ko2}. For the geometrically flavoured part Conant and Vogtmann use \textit{Outer space} \cite{cv}, a moduli space of marked graphs,\footnote{A \textit{marking} of a graph $G$ of rank $g$ is an identification of $\pi_1(G)$ with $F_g$ the free group on $g$ generators, up to homotopy; see \cref{rem:outer_space}.} as a model to ``realise" the associative, commutative and Lie graph complexes. 

`Model' and `realise' means here that we want to identify these complexes as (relative) cellular chain complexes of (pairs of) topological spaces. This representation is by clearly not unique. In fact, in the commutative and Lie case a simpler space, the moduli space of metric graphs\footnote{The modern lingo is ``(abstract) tropical curves" while ``metric graphs" were en vogue in the last century. Due to the recent return of the 90's in fashion and music I have decided to stick to the latter.} provides a natural universe in which these incarnations live \cite{cv,convogt,cgp}. The associative case does not quite fit into this story (in contrast to when working with Outer space); it requires a slightly different moduli space of (ribbon) graphs. However, the main idea and most constructions are almost identical. Therefore the same line of thought applies to this case as well.
\newline

\textbf{Disclaimer}: Almost everything that follows is in some form contained in the existing literature, at least implicitly. However, most works focus on one particular graph complex, and restrict to either its ``even" or ``odd" version. Furthermore, since there are so many variants of graph complexes and ways to define them (with a plethora of possible degree conventions etc.), it seems helpful to give a unified account, at least from the moduli space point of view---a concise introduction, covering and comparing all ways of defining and studying these complexes, is far beyond the scope of this paper. I tried to give as many references as possible, throughout the text and in particular in \cref{ss:background_and_related} below.

\begin{figure}[h]
    \centering
    \includegraphics[scale=0.35]{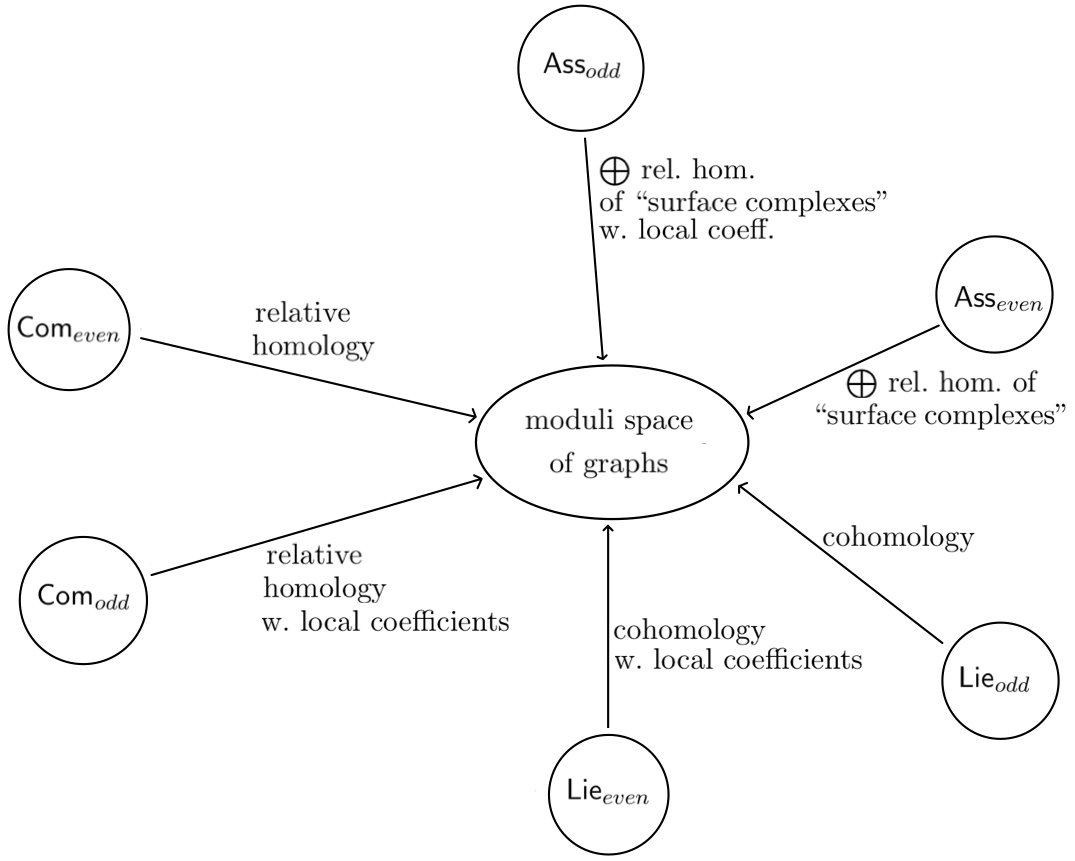}
     \caption{Relations between various graph complexes and the moduli space of graphs. An arrow $x \overset{y}{\to} z$ reads ``$x$ computes $y$ of $z$". Note that the ``surface complexes" in the associative case belong to a slightly different moduli space (see \cref{ss:gc_ass,ss:ass_cellular}).}
    \label{fig:GC_overview}
\end{figure}

\subsection{Outline of the paper}

In \cref{s:notation} we introduce some notation and review a couple of basic definitions from graph theory. With this at hand we define graph complexes in \cref{s:gc}: We start with a concise discussion of the notion of \hyperref[eq:orientation]{orientation}, both in the even and odd case, then define the \hyperref[ss:gc_comm]{commutative}, \hyperref[ss:gc_ass]{associative} and \hyperref[ss:gc_lie]{Lie} graph complexes. This includes a review of two \hyperref[sss:forested_graphs]{forested graph complexes} which are quasi-isomorphic to the latter. 

\Cref{s:moduli_space_of_graphs} introduces moduli spaces of (weighted) graphs $MG_g$ and certain subspaces thereof, most prominently the weight zero subspace $MG_g^0 \subset MG_g$ and its \textit{spine} $S_g$, a certain subcomplex of the barycentric subdivision of $MG_g$. We introduce ``natural" coordinates on $MG_g$ to prove that $MG_g^0$ deformation retracts onto $S_g$.

In \cref{s:MGandGC} we use this, and other cellular decompositions, to relate the homology of $MG_g$ to graph homology (see \cref{fig:GC_overview} for an illustration of these relations). For the even commutative and the odd Lie graph complexes we find
\begin{align*}
 H(\comm_{even}) & \cong \bigoplus_{g \in \NN} H(MG_g,MG_g\setminus MG_g^0)    , \\
 H(\lie_{odd}) & \cong \bigoplus_{g \in \NN}  H( MG_g^0)^\vee  ,
\end{align*}
up to some degree shifts.\footnote{We could here in fact work exclusively with the weight zero subspace $MG_g^0$, because $H(MG_g,MG_g\setminus MG_g^0)$ is isomorphic to the locally finite homology $H^{lf}(MG_g^0)$; see \cref{rem:locfin}.} 

After this connection is established, we use it in \cref{ss:odd_cellular} to explain how to treat $\comm_{odd}$ and $\lie_{even}$ with local coefficient systems on $MG_g$. In \cref{ss:simplify_cplxs} we use the geometric viewpoint once more to see how slightly different variants of the graph complexes (generated by other families of graphs) relate to each other (see \cref{thm:bigger_cplxs} and \cref{prop:def_retract_bivalent_edges}). This is known, but the ``standard" proof via spectral sequences is arguably more complicated.

We finish with a brief discussion of a geometric avatar for the associative graph complex.

\subsection{Background and related work} \label{ss:background_and_related}

Of course, everything is in some sense already contained in the two seminal papers \cite{ko1,ko2}. The geometric ideas outlined therein are expanded (and explained) in \cite{convogt}. 
For a similar but more algebraically flavoured overview see \cite{mahajan:symplectic}.

Other instances of the geometric viewpoint can be found in the following papers:
\begin{itemize}
    \item \cite{ConantGerlitsVogtmann:Cut} uses Outer space and relative homology with local coefficients to study the odd commutative graph complex.
    \item \cite{cgp,cgp2} studies the moduli space of tropical curves and identifies its homology with even commutative graph homology and the top weight cohomology of the moduli stack of curves $\m M_{g,n}$.
    \item \cite{Ww-Turchin:CommHairyGraphs} briefly discusses the moduli space of graphs (therein referred to as Outer space) as a geometric avatar for the commutative graph complex, and explains how its homology with local coefficients relates to group homology of $\mathrm{Out}(F_g)$---this includes the odd case (cf.\ \cref{ss:odd_cellular}).
\item \cite{brown21,brown22} constructs ``canonical" differential forms on the moduli space of graphs and uses an integration pairing with a variant of Stoke's theorem to detect homology classes in the even commutative graph complex.
\item \cite{ConantVogtmann:Morita2} discusses the family of \textit{Morita cycles} $m_k \in H_{4k}(\mathrm{Out}(F_{2k+2});\QQ)$ \cite{morita} in terms of a forested graph complex that originates from a cubical decomposition of Outer space. This is based on \cite{cv,hv}.
\end{itemize}

This list is by no means exhaustive, but should give a good starting point to look up details and further references. Most of what follows is based, more or less, on these works, in particular on \cite{convogt} and \cite{cgp}. 

Lastly, a good place to start learning about graph complexes from scratch are the lecture notes \cite{Willwacher:CPH} and \cite{Vogtmann:Intro_gc}.

\subsection{Acknowledgements}

I benefited from many illuminating discussions with Francis Brown and Karen Vogtmann on graph complexes and moduli spaces of graphs. Also many thanks to all participants of my lecture on graph complexes at HU Berlin in summer 2023.

\section{Conventions and preliminary definitions}\label{s:notation}

\subsection{Notation}

\begin{itemize}
    \item We write $x=(x_1,\ldots,x_d)$ for a vector $x \in k^d$ with $k\in \set{\mb R, \mb C}$ and $[x]=[x_1:\ldots:x_d]$ for the corresponding point in $\mb P(k^d)$ -- we sometimes omit the brackets $[ \cdot ]$ if the meaning is clear from the context. We set $\mb R_+=(0,\infty)$.
    \item If $I$ is a finite set and $k$ a field, we denote by $k^I$ the $|I|$-dimensional vector space with fixed basis the elements of $I$. 
    \item Unless denoted otherwise, all homology groups will be with rational coefficients. 
    \item If no confusion is possible, we write $a$ for a singleton $\set a$.
    \item The symbols $ \dot{\in} $ and $ \dot{\subseteq} $ mean ``\ldots pairwise different elements / pairwise disjoint subsets of \ldots ". 
\end{itemize}

\subsection{Graphs}

Graphs appear in many areas of maths and physics which often use very different terminology. We will use the following conventions:

\begin{itemize}
    \item A \textit{graph} is a tupel $G=(V_G,H_G,\iota,\epsilon)$ where $V_G$ is the set of vertices of $G$, and $H_G$ is the set of \textit{half-edges} of $G$. The map $\iota \colon H_G \to V_G$ connects half-edges to vertices. We call $\val v =|\inv{\iota}(v)|$ the \textit{valence} of $v\in V_G$. The map $\epsilon$ is an involution on $H_G$: If $\epsilon(h_1)=h_2$, then the pair $\set{h_1,h_2}$ specifies an (internal) edge of $G$. If $\epsilon(h)=h$, then we call $h$ a \textit{leg} (or external edge) of $G$. We will often omit this data and simply write $G=(V_G,E_G)$ for a graph prenteding that $G$ is specified by its sets of vertices and (internal) edges. We set $e_G=|E_G|$, $v_G=|V_G|$, and we write $h_G=\dim \homgra$ for the \textit{rank} or \textit{loop number} of $G$.
    \item  A \emph{tadpole} (self-loop) is an edge formed by two half-edges that are connected to the same vertex.
    \item A graph is \textit{directed}\footnote{This is usually called an orientation on $G$, but we will use this notion differently.} if for each edge $e \in E_G$ an order on $e=\set{h_1,h_2} $ is specified. If $e=(h_1,h_2)$ then we call $\iota(h_1) \in V_G$ the \textit{source} $s(e)$ of $e$ and $\iota(h_2) \in V_G$ the \textit{target} $t(e)$ of $e$.
    \item All graphs we consider here will be finite and connected. 
    \item A graph is \textit{core, bridge-free} or \textit{1-particle irreducible} if removing any edge reduces its rank by one. For connected graphs with all vertices of valence at least 2 this is equivalent to having no \textit{bridges}, i.e., edges whose removal disconnects the graph.
    \item A \emph{tree} is a graph $T$ such that $h_T=0$. A \emph{forest} is a disjoint union of trees.
    \item A \emph{subgraph} $\gamma \subset G$ (without legs) is a graph $\gamma$ such that $V_\gamma\subset V_G$ and $E_\gamma \subset E_G$. If $\gamma\subset G$ is a subgraph, we write $G\setminus\gamma$ for the subgraph of $G$ defined by $V_{G\setminus\gamma}=V_G$ and $E_{G\setminus\gamma}=E_G \setminus E_\gamma$. We write $G/\gamma$ for the graph obtained from $G$ by collapsing each connected component of $\gamma$ to a vertex. We call such quotient graphs \textit{cographs}.
    \item A \emph{spanning tree} of $G$ is a subgraph $T \subset G$ such that $T$ is a tree and $V_T=V_G$. A \emph{spanning forest} of $G$ is a subgraph $F \subset G$ such that $F$ is a forest and $V_F=V_G$.
    \item We sometimes abuse notation by identifying edge sets and subgraphs in the obvious way.
\end{itemize}

\subsection{Graph morphisms}

Let $G_i=(V_i,H_i,\iota_i,\epsilon_i)$, $i=1,2$, be two graphs. A \textit{graph morphism} $\varphi \colon G_1 \to G_2$ is a pair of maps 
\[
 \varphi=(\varphi_V,\varphi_H), \quad    \varphi_V \colon V_1 \to V_2, \quad  \varphi_H \colon H_1 \to H_2,
\] 
such that for any half-edge $h \in H_1$ we have $\iota_2( \varphi_H(h) ) = \varphi_V(\iota_1(h))$ and $\varphi_H(\epsilon_1(h))=\epsilon_2(\varphi_H(h))$. 

A graph morphism $\varphi\colon G_1 \to G_2$ is a \textit{graph isomorphism} if there exists a graph morphism $\psi\colon G_2 \to G_1$ with $\varphi \circ \psi = \id_{G_2} $ and $\psi \circ \varphi = \id_{G_1}$. Equivalently, $\varphi$ is an isomorphism if both $\varphi_V$ and $\varphi_H$ are bijective and $\varphi$ has a left- or right-inverse.

Note that any graph morphism $\varphi \colon G_1 \to G_2$ induces a map $\varphi_E$ on the edge sets of $G_1$ and $G_2$. We will be mainly interested in this map. To keep notation light we often abuse notation and simply denote it by $\varphi$. 

The group of automorphisms $\varphi \colon G \overset{\sim}{\to} G$ is denoted by $\Aut(G)$. Note that $G$ may have automorphisms that act trivial on $E_G$, for instance $\Aut( \! \! \!\! \scalebox{1.4}{\tad} \! \! \! )\cong \ZZ_2$.

\section{Graph complexes}\label{s:gc}

We give a short and simple definition of ``the" three\footnote{Many other variants are possible, see e.g.\ \cite{Willwacher:CPH}.} graph complexes. References for more details, equivalent definitions, and further readings are given along the way; see also the list of references in \cref{ss:background_and_related}. 
\newline

Let $G$ be a finite, connected graph and let $N \in \NN$. Define the \textbf{degree} of $G$ by 
\[
\deg_N G = e_G(1-N) + (v_G-1)N.
\]
If $G$ is connected, then $\deg_NG = e_G - Nh_G$.
\newline

For a vector space $V$ of dimension $n$ let $\det V=\bigwedge^n V$. Note that an element of $(\det V)^\times$ specifies an orientation of $V$.

\begin{defn}

An \textbf{orientation} of a graph $G$ is an orientation of the vector space $\QQ^{E_G}  \oplus \homgra$ where both $\QQ^{E_G}$ and $\homgra$ are considered as graded vector spaces, generated by elements of degree one and $N$, respectively. We write  
\be \label{eq:orientation}
  \ori \in \left(\det \QQ^{E_G}  \otimes \det \homgra \right)^\times ,
\ee
for an orientation of $G$.
\end{defn}

In other words, an orientation on $G$ is an order on $E_G$, up to even permutations, together with an orientation of the $\QQ$-vector space of its independent cycles. Since we view both as graded vector spaces, we have for $\sigma \in S_n$
\[
e_{\sigma(1)} \wedge \ldots \wedge e_{\sigma(n)} = \mathrm{sgn}(\sigma) e_1 \wedge \ldots \wedge e_n 
\]
and
\[
c_{\sigma(1)} \wedge \ldots \wedge c_{\sigma(n)} = \mathrm{sgn}(\sigma)^{N} c_1 \wedge \ldots \wedge c_n. 
\]
Hence, if $N$ is even, then $\det \QQ^{E_G} $ specifies an orientation on $G$.

In the odd case it will be useful to have an explicit representation of $\homgra$ (to compute $\det \homgra$). For this we pick a spanning tree $T \subset G$ and direct each edge in the complement of $T$. Then each $e \in E_{G\setminus T}$ represents an oriented cycle that starts along $e$ in the given direction and runs through $T$ along the unique minimal edge-path in $T$ that connects $t(e)$ with the starting vertex $s(e)$. This gives a basis of $\homgra$ which can be oriented by choosing an order on the representing edges in $G\setminus T$. To change $T$ we iteratively exchange edges from $T$ and $G\setminus T$. In each step $E_{G\setminus T} \ni e_1 \leftrightarrow e_2 \in E_T$, the orientation of $e_1$ determines an orientation of $e_2$ (by the mechanism described above---note that $e_2$ must lie on the edge-path specified by $e_1$), and the order on the cycles stays the same (the cycle represented by $e_1$ replaces the cycle represented by $e_2$).

\begin{rem}\label{rem:simplicial_orient}
     The above recipe works for singular homology. Note that if we want to compute the homology of $G$ as a simplicial or CW-complex, we need to direct \emph{every} edge. Only then can we translate the above representation into the usual representation of cycles as (ordered) $\ZZ$-linear combinations of edges. Of course, $\homgra$ is independent of this choice, but $\det \homgra$ depends on this choice, up to an even number of ``direction changes". 
\end{rem}

\begin{example}
Let $G$ be the theta graph depicted on the left hand side of \cref{fig:theta}. We order $E_G$ by $e_1<e_2<e_3$. The graph in the middle encodes a representation $\homgra=\langle c_1,c_2 \rangle$ where $c_i$ travels from $v_1$ to $v_2$ and back to $v_1$ along $e_i$ and $T=e_3$ (if all $e_i$ are directed from $v_1$ to $v_2$ (in the sense of \cref{rem:simplicial_orient}), then $c_1= e_1-e_3$ and $c_2=e_2-e_3$). The numbers over the directed edges encode the order, that is, the orientation on $\homgra$ is given by $c_1 \wedge c_2$. If we choose a different tree $T$, e.g.\ $T=e_2$, we get a new ordered basis $(c_1',c_2')$: The cycle $c_1'$ runs from $v_1$ to $v_2$ along $e_1$ and then back to $v_1$ via $e_2$, the cycle $c_2'$ runs along $e_3$ from $v_2$ to $v_1$ and then back to $v_2$ via $e_2$ (hence $c_1'=e_1-e_2$ and $c_2'=c_2$). In terms of orientations,
\[
\ori   = (e_1 \wedge e_2 \wedge e_3) \otimes (c_1 \wedge c_2) =  (e_1 \wedge e_2 \wedge e_3) \otimes (c_1' \wedge c_2').
\]
Note that the ``opposite" orientation is given by reversing the direction of one of the edges that specify the cycles. For instance, if we flip the direction on $e_1$ in the middle graph, then this represents the basis $(-c_1,c_2)$. This is the same as swapping the order of the $c_i$, since $-c_1 \wedge c_2 = c_2 \wedge c_1$. The same holds for the $c_i'$.
\end{example} 

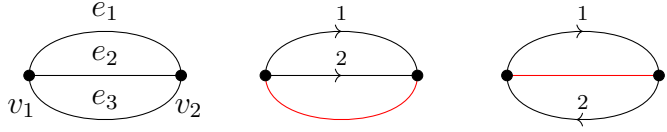
\begin{figure}[h]
   \begin{tikzpicture}[scale=1]
\coordinate (v0) at (0,0) node {};
  \coordinate  (v1) at (0,2) node[above,xshift=-.1cm,yshift=1.3cm] {$v_1$};
   \coordinate  (v2) at (2,2) node[above,xshift=2.1cm,yshift=1.3cm] {$v_2$};
   \draw (v1) to[out=90,in=90] node[above] {$e_1$} (v2);
   \draw (v1) -- node[above] {$e_2$} (v2);
   \draw (v1) to[out=-90,in=-90] node[above] {$e_3$} (v2);
  \filldraw[fill=black] (v1) circle (0.07);
  \filldraw[fill=black] (v2) circle (0.07);
  \end{tikzpicture}
  \quad 
   \begin{tikzpicture}[scale=1]
\coordinate (v0) at (0,0) node {};
  \coordinate  (v1) at (0,2);
   \coordinate  (v2) at (2,2);
   \draw[decoration={markings, mark=at position 0.5 with {\arrow{>}}},postaction={decorate}] (v1) to[out=90,in=90] node[above] {$\scriptstyle{1}$} (v2);
   \draw[decoration={markings, mark=at position 0.5 with {\arrow{>}}},postaction={decorate}] (v1) -- node[above] {$\scriptstyle{2}$}  (v2);
   \draw[red] (v1) to[out=-90,in=-90] (v2);
  \filldraw[fill=black] (v1) circle (0.07);
  \filldraw[fill=black] (v2) circle (0.07);
  \end{tikzpicture}
  \qquad
   \begin{tikzpicture}[scale=1]
\coordinate (v0) at (0,0) node {};
  \coordinate  (v1) at (0,2);
   \coordinate  (v2) at (2,2);
   \draw[decoration={markings, mark=at position 0.5 with {\arrow{>}}},postaction={decorate}] (v1) to[out=90,in=90] node[above] {$\scriptstyle{1}$} (v2);
   \draw[red] (v1) -- (v2);
   \draw[decoration={markings, mark=at position 0.5 with {\arrow{<}}},postaction={decorate}] (v1) to[out=-90,in=-90] node[above] {$\scriptstyle{2}$} (v2);
  \filldraw[fill=black] (v1) circle (0.07);
  \filldraw[fill=black] (v2) circle (0.07);
  \end{tikzpicture}
    \caption{The theta graph and two bases for $\homgra$ represented by the choice of a spanning tree $T\subset G$ and directions on edges in the complement of $T$.}
    \label{fig:theta}
\end{figure}

\begin{rem}\label{rem:odd_orient}
    If $N$ is odd, then specifying an orientation of $G$ is equivalent to ordering its vertices and orienting each edge, both up to even permutations; see \cite[\S 2.3.1]{convogt}. This variant is much easier to handle when computing differentials (\cref{eq:orient_collapse_odd} below). However, our notion will become useful later in \cref{s:MGandGC}.
\end{rem}

\subsection{The commutative graph complex $\comm_N$} \label{ss:gc_comm}
 
Let $N\in \NN$.
Let $\comm_N$ be the $\QQ$-vector space spanned by linear combinations of pairs $(G,\ori)$ where $G$ is connected, has neither tadpoles nor vertices of valence less than three, and $\eta$ is an orientation of $G$, subject to the relations
\begin{align}
  \label{eq:orient_i}& (G,-\ori)  = -(G,\ori),\tag{$i$} \\ 
 \label{eq:orient_ii} & (G,\ori) = (G', \varphi^* \ori) \text{ for any isomorphism } \varphi \colon G' \overset{\sim}{\to} G. \tag{$ii$}
\end{align}
Write $[G,\ori]$ for the equivalence class of $(G,\ori)$. We abbreviate it by $[G]$, if a specific choice of orientation is not important.

Sometimes it is useful to work in larger complexes. For this we define $\comm_N^{\geq2}$, $\comm_N^{\tad}$ and $\comm_N^{\geq2,\tad}$ as variants of $\comm_N$ which allow for vertices of valence 2, tadpoles, and both, respectively.

\begin{lem}
    If $N$ is even and $\varphi \in \Aut(G)$ induces an odd permutation on $E_G$, then $[G]=0$. In particular, all graphs with multi-edges vanish in $\comm_N$ for even $N$. 
    If $N$ is odd, then all graphs with tadpoles vanish in $\comm_N^{\tad}$.
\end{lem}

 \begin{proof}
 For the odd case observe that reversing the direction of a tadpole is an automorphism of $G$ that induces an orientation-reversing map on $\homgra$. In both cases, given an ``odd" automorphism $\varphi$ we have $(G,\ori) \overset{(ii)}{=} (G, \varphi^*\ori) \overset{(i)}{=}-(G,\ori )$ and therefore $ [G,\ori]=0.$
\end{proof}

\begin{figure}[h]
\begin{tikzpicture}
\coordinate (c) at (0,0);
    \fill (c) circle (0.06cm);
  \draw (c) circle (1.15cm);
\foreach \n in {0,...,5} \coordinate (k\n) at (90+360/5*\n:1.15);
\foreach \n in {0,...,5} \draw (c) -- (k\n);
\foreach \n in {0,...,5} \filldraw[black] (k\n) circle (0.05cm);
\end{tikzpicture}
\caption{The wheel $W_5$ with 5 spokes.}
    \label{fig:wheel}
\end{figure}
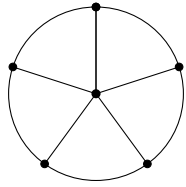

\begin{example}\label{eg:vanishing}
Consider three families of graphs, cycles $C_n$ on $n$ vertices/edges, wheels $W_n$ with $n$ spokes (see \cref{fig:wheel}), and bananas $B_n$ (two vertices connected by $n$ edges). 

 If $N$ is even, then $[G]=0$ for $G \in \set{W_{2n},B_n \mid n\geq 1}$ or for $G=C_n$ with $n\equiv 0,2,3 \mod 4$ since all these graphs have odd symmetries. On the other hand, the graphs $W_{2n+1}$ and $C_{4n+1}$ for $n\geq 1$ have only even symmetries, hence represent non-zero classes in $\comm_N$.

 If $N$ is odd, then from \cref{rem:odd_orient} we immediately deduce that $[B_n]=0$ if and only if $n$ is even. For the cycles we find that $C_n$ is non-zero if and only if $n\equiv 3 \mod 4$. Each even wheel $W_{2n}$ has a reflection symmetry which induces an even permutation on $E_{W_{2n}}$ and is orientation-reserving on $H_1(W_{2n})$, hence $[W_{2n}]=0$. The same argument works for the odd wheels; all symmetries are even on the edges and orientation-preserving on $H_1(W_{2n+1})$, except for the reflections which swap $2n$ cycles. Their sign depends on the parity of $n$. We conclude that $[W_{2n+1}]=0$ if and only if $n$ is even.
\end{example}

We equip $\comm_N$ with a differential $d \colon \comm_N \to \comm_N$ which is defined by collapsing edges. For this let 
\[
G\sla e = \begin{cases}
    0 & \text{ if $e$ is a tadpole,} \\
    G/e & \text{ else.}
\end{cases} 
\]
Define $d$ as the linear map that acts on generators by
\be \label{eq:d_in_comm}
d [G,\ori] = \sum_{e \in E_G} [G\sla e, \eta_{G\sla e}],
\ee
where $\eta_{G\sla e}$ is obtained from $\ori$ as follows:
If
\[
\ori = e_1 \wedge \cdots \wedge e_n \otimes c_1 \wedge \cdots \wedge c_g,
\]
then 
\be \label{eq:orient_collapse_odd}
\eta_{G\sla e_i} = (-1)^{i} e_1 \wedge \cdots \wedge \widehat{e_i} \wedge \cdots \wedge e_n \otimes \xi_e^*(c_1 \wedge \cdots \wedge c_g)
\ee
where $\xi_e \colon H_1(G/e) \overset{\sim}{\to} \homgra$ is the isomorphism on homology induced by collapsing the (non-tadpole) edge $e$. Concretely, this means we choose a representation of $\homgra$ (and hence of $\ori$) with $e\in T$, then $\xi_e^*(c_1 \wedge \cdots \wedge c_g)=c_1 \wedge \cdots \wedge c_g$.\footnote{To find the induced orientation on $G/e$ in terms of \cref{rem:odd_orient} we permute the vertex order so that $e$ points from the first to the second vertex, then collapse $e$ and keep the order and direction on the remaining vertices and edges.} Recall that for $N$ even only the order on $E_G$ matters, so in this case it suffices to write
\be \label{eq:orient_collapse_even}
\eta_{G\sla e_i} = (-1)^{i} e_1 \wedge \cdots \wedge \widehat{e_i} \wedge \cdots \wedge e_n. 
\ee

A simple counting argument shows
\begin{lem}
$d^2=0$.
\end{lem}

The chain complex $(\comm_N,d)$ is called the (even/odd) \textbf{commutative graph complex} and its homology is the (even/odd) \textbf{commutative graph homology},
\[
H(\comm_N)=\ker d / \image{d}.
\]
It is bigraded by the degree $\deg_N$ and the loop-number $h_G$. Note that $d$ does not change the loop-number of a graph, hence $\comm_N$ splits into a direct sum of complexes
\[
\comm_N= \bigoplus_{g \geq 1} \comm_N^{(g)},
\]
where $\comm_N^{(g)}$ is the subcomplex generated by graphs with $g$ loops.

All complexes $\comm_N$ for $N$ even are isomorphic, and likewise for $N$ odd. There is no known relation between the even and odd case.

\begin{rem}
 One can also define \textbf{graph cohomology} by viewing $\comm_N$ as a cochain complex with a differential $d^\vee$ that adds edges in all possible ways. More algebraically phrased, there is a Lie algebra structure on $\comm_N$ where $[G,H]$ is defined by summing over all ways of inserting $G$ into $H$ and vice versa. The differential $d^\vee$ arises then from the Lie bracket with the graph consisting of a single edge, $d^\vee(G)= [G,e]$. For details we refer to \cite{Willwacher:GRT}; see also \cite{kwc-dogc}.  
\end{rem}

We can also consider graph homology of larger complexes $\comm_N^i$, $i\in \set{\geq2,\tad,\set{\geq2,\tad}}$, that allow for graphs with bivalent vertices, with tadpoles, and both, respectively. However, the following theorem shows that all the interesting homology is concentrated in $\comm_N$.

\begin{thm}[cf.\ \cite{ko1,ko2,Willwacher:GRT}] \label{thm:bigger_cplxs}
The various commutative graph homologies are related by 
    \begin{equation*}
         H(\comm_N^{\geq2}) = H(\comm_N) \oplus \begin{cases}
             \bigoplus_{n>0} \QQ[C_{4n+1}] & \text{ if $N$ is even,} \\
             \bigoplus_{n\geq0} \QQ[C_{4n+3}] & \text{ if $N$ is odd.}
         \end{cases} 
    \end{equation*} 
If $N$ is odd, then $H(\comm_N^{\tad}) = H(\comm_N)$. For $N$ even they differ by a single additional class in degree $1-N$, generated by $[ \!\begin{tikzpicture}[baseline=-1,scale=1, every loop/.style={}] \coordinate  (v) at (0,0);
    \filldraw[fill=black] (v) circle (0.035);
    \draw (v) edge[loop] (v);
\end{tikzpicture}\!]$.
Therefore 
\begin{equation*}
  H(\comm_N^{\tad\!,\geq2})= H(\comm_N) \oplus  \begin{cases}
             \bigoplus_{n\geq0} \QQ[C_{4n+1}] & \text{ if $N$ is even,} \\
             \bigoplus_{n\geq0} \QQ[C_{4n+3}] & \text{ if $N$ is odd.}
         \end{cases} 
\end{equation*}
\end{thm}

\begin{proof}
    From \cref{eg:vanishing} it is clear that the cycles $C_{4n+1}$ represent non-trivial homology classes. To show the first statement, note that $\comm_N^{\geq2}$ splits into a direct sum of complexes
    \[
    \comm_N^{\geq2}= \comm_N^{\neq2} \oplus \comm_N^{=2}
    \]
    where $\comm_N^{\neq2}$ consists of graphs with at least one vertex of valence 3 and $\comm_N^{=2}=\QQ \langle [C_n] \mid n>0 \rangle$. 
    
    In \cref{ss:simplify_cplxs} we will argue geometrically that $\comm_N^{\neq2}$ is quasi-isomorphic to $\comm_N$. Here we sketch the algebraic argument. 
    
    The complex $\comm_N^{\neq2}$ splits further into $\comm_N \oplus \comm_N^{2,3}$ where the latter consists of graphs with at least one vertex of valence 2 and at least one of valence 3. This is indeed a subcomplex since if $G$ has only a single vertex of valence 2, connected to $e_1$ and $e_2$, then $G/e_1$ cancels $G/e_2$ in $dG$ so that $dG\in \comm_N^{2,3}$. One can use the following model to show that this complex is acyclic: Any vertex of valence 2 lies on a unique edge-path connecting two vertices of valence $\geq 3$. Replace in each graph in $\comm_N^{2,3}$ these edge-paths by edges labeled by the number of vertices in the path. Let $T$ denote the total complex associated to a double complex spanned by such labeled graphs where one differential $d_u$ collapses unlabeled edges and the other $d_\ell$ acts on labeled edges by sending an edge with label $k$ to zero if $k$ is odd, and to $k-1$ if $k$ is even. Now set up a spectral sequence whose first page differential is $d_\ell$. It converges to zero, hence $H(\comm_N^{2,3})=0$.
    
    For the second (and third) statement a similar spectral sequence argument applies.\footnote{This also has a geometric variant, but the argument is quite involved; see \cite[\S 4.2]{cgp2}.} $\comm_N^{\tad}$ splits into a direct sum of subcomplexes, $\comm_N$ and a complex of graphs with tadpoles. The latter is acyclic: A graph with tadpoles consists of a base graph (with vertices of valence $\geq2$) to which some ``antennas" are attached; here an antenna is a rooted tree with at least trivalent vertices and leaves decorated by tadpoles. Now model the subcomplex generated by tadpole graphs by a double complex $T$ with differentials $d_a$ (collapse inner edges of antennas) and $d_b$ (collapse base edges). Filter $T$ by number of base edges. Then the associated spectral sequence has trivial first page, hence converges to zero. 
\end{proof}

The overall structure of $H(\comm_N)$ is not well understood. Euler characteristic computations \cite{WZ:euler_char} show that there are many non-trivial classes, but only a few concrete examples are known. Willwacher showed in \cite{Willwacher:GRT} that $H^0(\comm_2)\cong \mathfrak{grt}$ with $\mathfrak{grt}$ denoting the Grothendieck-Teichmüller Lie algebra. The right hand side contains a free Lie algebra  on generators $\sigma_{2n+1},n >0$. Their representatives in graph cohomology pair non-trivially with the wheels $W_{2n+1}$ (see \cite[\S 9]{Willwacher:GRT}, also \cite[\S1.3]{brown21}). 

\begin{rem}\label{rem:1vi}
  A graph is \textit{1-vertex irreducible} if it stays connected after removing any vertex together with its adjacent edges. It is shown in \cite{ConantGerlitsVogtmann:Cut} that $\comm_N$ is quasi-isomorphic to the quotient $\comm_N/\comm_N^{1vr}$ where $\comm_N^{1vr}$ denotes the subcomplex of 1-vertex \textit{reducible} graphs. Note that the latter contains the subspace generated by graphs with bridges, since if $G$ has a bridge $e=\set{v,w}$, then removing either $v$ or $w$ disconnects $G$.
\end{rem}

\subsection{The associative graph complex $\ass_N$} \label{ss:gc_ass}

The associative case is very similar to the commutative case, except that we put additional structure on graphs, specifying at every vertex a cyclic ordering of the incident edges. 

A \textbf{ribbon} or \textbf{fat graph} is a graph  
$G=(V_G,H_G,\iota,\epsilon,\set{\sigma_v}_{v\in V(G)})$ where $\sigma_v$ is a cyclic order of $\inv{\iota}(v)\subset H(G)$.
Define degree and orientations as in the commutative case and let $\ass_N$ be the rational vector space spanned by linear combinations of ribbon graphs without tadpoles and with all vertices of valence at least three, subject to the relations
\eqref{eq:orient_i} and  \eqref{eq:orient_ii}---the latter with respect to isomorphisms of ribbon graphs.

The differential $d \colon  \ass_N \to \ass_N$ is defined as above where the ribbon structure of $G\sla e$ is obtained as follows: Let $e \in E_G$ be specified by $\epsilon(h)=h'$ for $ h,h' \in H_G$, and let $\epsilon(h)=v$, $\epsilon(h')=w$. If $\sigma_v=(h_{i_1} \cdots h_{i_s})$ and $\sigma_w=(h_{j_1} \cdots h_{j_k})$ such that $h_{i_s}=h$ and $h_{j_1}=h'$, then the ribbon structure at the new vertex $u\in V_{G\sla e}$ which is the image of $v$ and $w$ under the collapse of $e$ is
\[
\sigma_{u}=(h_{i_1} \cdots h_{i_{s-1}} h_{j_2} \cdots h_{j_k}).
\]
One checks that $d^2=0$. 

The \textbf{associative graph complex} is the complex $(\ass_N,d)$ and \textbf{associative graph homology} is defined as $H(\ass_N)=\ker d / \image{d}$. It is bigraded by degree and loop-number, the latter being invariant under the differential $d$.

A ribbon graph $G$ can be ``fattened" to an oriented surface $S_G$ (\cite{Penner:Fatgraphs}, see also \cite[\S 4.1.1]{convogt}). This surface is unique (up to homeomorphism) among compact, connected, oriented surfaces $S$ with $\pi_1(S)=\pi_1(G)$ that admit an embedding $f \colon G \to S$ such that the ribbon structure of $G$ is induced by the orientation of $S$. See \cref{fig:fat_graphs} for an example. Note that collapsing an edge does not change this property, hence $S_G \cong S_{G/e}$ for each (non-tadpole) $e\in E(G)$. It follows that 
\[
\ass_N = \bigoplus_{g\geq 1} \ass_N^{(g)} = \bigoplus_{g\geq 1} \bigoplus_{ \\ \pi_1(S)=F_g } \ass_N^{(g,S)}
\]
where the second sum runs over homeomorphism classes of compact, connected oriented surfaces $S$ with fundamental group isomorphic to the free group on $g$ generators. Each complex $\ass_N^{(g,S)}$ computes the cohomology of the mapping class group $\mathrm{Mod}(S)$ of $S$ \cite[\S4.2]{convogt}. 

\begin{figure}[h]
  \hspace{1.8cm} \begin{tikzpicture}[scale=1]
\coordinate (v0) at (0,0);
  \coordinate  (v1) at (0,2);
   \coordinate  (v2) at (2,2);
   \draw (v1) node[above,yshift=.2cm] {$\scriptscriptstyle{1}$} to[out=90,in=90] (v2) node[above,yshift=.2cm] {$\scriptscriptstyle{1}$};
   \draw (v1) node[above,xshift=.2cm,yshift=-.1cm] {$\scriptscriptstyle{3}$} -- (v2) node[above,xshift=-.2cm,yshift=-.1cm] {$\scriptscriptstyle{2}$};
   \draw (v1) node[below,yshift=-.2cm] {$\scriptscriptstyle{2}$} to[out=-90,in=-90]  (v2) node[below,yshift=-.2cm] {$\scriptscriptstyle{3}$};
  \filldraw[fill=black] (v1) circle (0.06);
  \filldraw[fill=black] (v2) circle (0.06);
  \end{tikzpicture}
  \hspace{1.5cm} 
     \begin{tikzpicture}[scale=1]
\coordinate (v0) at (0,0);
  \coordinate  (v1) at (0,2);
   \coordinate  (v2) at (2,2);
   \draw (v1) node[above,yshift=.2cm] {$\scriptscriptstyle{1}$} to[out=90,in=90] (v2) node[above,yshift=.2cm] {$\scriptscriptstyle{1}$};
   \draw (v1) node[above,xshift=.2cm,yshift=-.1cm] {$\scriptscriptstyle{3}$} -- (v2) node[above,xshift=-.2cm,yshift=-.1cm] {$\scriptscriptstyle{3}$};
   \draw (v1) node[below,yshift=-.2cm] {$\scriptscriptstyle{2}$} to[out=-90,in=-90]  (v2) node[below,yshift=-.2cm] {$\scriptscriptstyle{2}$};
  \filldraw[fill=black] (v1) circle (0.06);
  \filldraw[fill=black] (v2) circle (0.06);
  \end{tikzpicture}
  \newline
  
  \includegraphics[scale=.23]{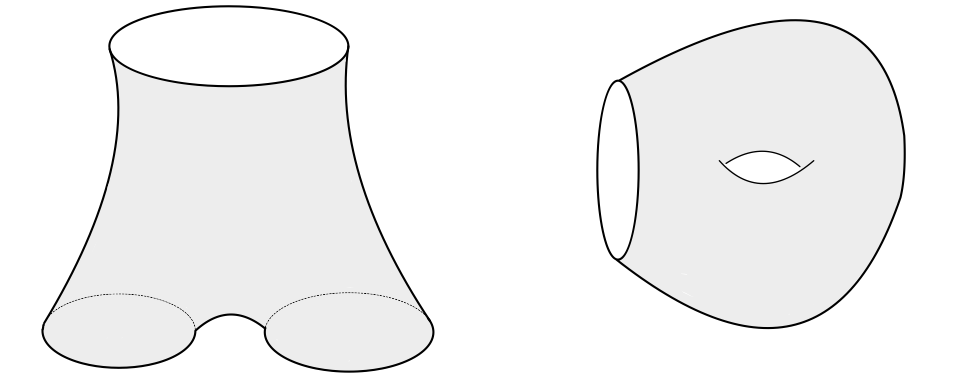}
    \caption{Two ribbon graphs (half-edges ordered by $1<2<3$) and their associated surfaces.}
    \label{fig:fat_graphs}
\end{figure}

\subsection{The Lie graph complex $\lie_N$} \label{ss:gc_lie}

Here the story is a little more complicated than in the previous two cases. One considers graphs with vertices decorated by elements of a certain Lie algebra.
We sketch the construction of these so-called $\mathrm{Lie}$-graphs and of the associated graph complex $\lie_N$, then introduce a different complex (of pairs of graphs) whose homology is closely related to that of $\lie_N$. 

\subsubsection{$\mathrm{Lie}$-decorated graphs}

Consider the free vector space generated by finite connected oriented graphs $G$ where each vertex $v\in V_G$ is decorated by a planar binary tree with $n=\val v$ labeled leaves (so that the set of its leaves is in one-to-one correspondence to $\inv{\iota}(v)$). 
Define $\lie_N$ as the quotient space by the usual relations, \eqref{eq:orient_i}, \eqref{eq:orient_ii}, and two additional relations, 
\begin{itemize}
    \item \label{as} antisymmetry: Reversing the (planar) orientation of a tree's vertex produces a minus sign, 
    \[
    \begin{tikzpicture}[scale=.4]
\coordinate (r) at (0,0);
  \coordinate  (v1) at (0,1);
   \coordinate  (v2) at (.5,1.5);
   \coordinate  (l1) at (-1,2);
   \coordinate  (l2) at (0,2);
   \coordinate  (l3) at (1,2);
  \filldraw[fill=black] (v1) circle (0.06);
   \filldraw[fill=black] (v2) circle (0.06);
      \draw (r) -- (v1);
      \draw (v1) -- (l1);
      \draw (v1) -- (v2);
      \draw (v2) -- (l2);
      \draw (v2) -- (l3);
  \end{tikzpicture}
  = -  \begin{tikzpicture}[scale=.4]
\coordinate (r) at (0,0);
  \coordinate  (v1) at (0,1);
   \coordinate  (v2) at (-.5,1.5);
   \coordinate  (l1) at (1,2);
   \coordinate  (l2) at (0,2);
   \coordinate  (l3) at (-1,2);
  \filldraw[fill=black] (v1) circle (0.06);
   \filldraw[fill=black] (v2) circle (0.06);
      \draw (r) -- (v1);
      \draw (v1) -- (l1);
      \draw (v1) -- (v2);
      \draw (v2) -- (l2);
      \draw (v2) -- (l3);
  \end{tikzpicture} .
  \]
    \item \label{IHX}Jacobi aka IHX identity:
    \[ 
 \begin{tikzpicture}[scale=.5]
\coordinate (u) at (0,0);
  \coordinate  (o) at (0,.5);
   \coordinate  (lo) at (-.5,1);
   \coordinate  (ro) at (.5,1);
   \coordinate  (lu) at (-.5,-.5);
   \coordinate  (ru) at (.5,-.5);
  \filldraw[fill=black] (o) circle (0.06);
   \filldraw[fill=black] (u) circle (0.06);
      \draw (u) -- (o);
      \draw (u) -- (lu);
      \draw (u) -- (ru);
      \draw (o) -- (ro);
      \draw (o) -- (lo);
  \end{tikzpicture}  
 \ - \
 \begin{tikzpicture}[scale=.5]
\coordinate (l) at (-.25,.25);
  \coordinate  (r) at (.25,.25);
   \coordinate  (lo) at (-.5,1);
   \coordinate  (ro) at (.5,1);
   \coordinate  (lu) at (-.5,-.5);
   \coordinate  (ru) at (.5,-.5);
  \filldraw[fill=black] (l) circle (0.06);
   \filldraw[fill=black] (r) circle (0.06);
      \draw (l) -- (r);
      \draw (l) -- (lu);
      \draw (l) -- (lo);
      \draw (r) -- (ro);
      \draw (r) -- (ru);
  \end{tikzpicture}  
 \ + \
   \begin{tikzpicture}[scale=.5]
\coordinate (l) at (-.17,.25);
  \coordinate  (r) at (.17,.25);
   \coordinate  (lo) at (-.5,1);
   \coordinate  (ro) at (.5,1);
   \coordinate  (lu) at (-.5,-.5);
   \coordinate  (ru) at (.5,-.5);
  \filldraw[fill=black] (l) circle (0.06);
   \filldraw[fill=black] (r) circle (0.06);
      \draw (l) -- (r);
      \draw (l) -- (lu);
      \draw (l) -- (ro);
      \draw (r) -- (lo);
      \draw (r) -- (ru);
  \end{tikzpicture}
 \ =0.
\]
\end{itemize}

The differential is the same as for the previous two complexes with the tree decoration of $G\sla e$ defined as follows. If $e=(v,w)=(h_1,h_2)$ with $v$ and $w$ decorated by trees $T_v$ and $T_w$, then the image of $e$ in $G/e$ is decorated by the tree $T_v*_{h_1,h_2}T_w$ obtained by identifying the leaf $h_1$ of $T_v$ with the leaf $h_2$ of $T_w$ (they become then an inner edge of $T_v*_{h_1,h_2}T_w$ which is thus again a binary tree with $\val v +\val w-2$ leaves). 

The resulting complex $(\lie_N,d)$ is called the \textbf{Lie graph complex} and its homology $H(\lie_N)$ is the \textbf{Lie graph homology}.

Denote the subcomplex generated by graphs of rank $g$ by $\lie_N^{(g)}$. Kontsevich shows in \cite{ko1} (see also \cite{convogt} and note that both use a different grading---by number of vertices)
\begin{thm}\label{thm:Lie_and_Out}
  $H_\bullet(\lie_N^{(g)}) \cong H^{(3-N)g-3-\bullet}(\mathrm{Out}(F_g);\QQ_N')$ where $\QQ_N'$ is the trivial representation if $N$ is odd, and the $\det$-representation\footnote{Let $\alpha$ denote the map $\mathrm{Out}(F_g) \to \mathrm{GL}_g(\ZZ)$. Then $\mathrm{Out}(F_g)$ acts on $\QQ$ via $x.q= \det\alpha(x)\cdot q$.} if $N$ is even. 
\end{thm}

The situation is similar to the commutative case; apart from a few explicit examples not much is known about the structure of Lie graph homology. Morita \cite{morita} constructed an infinite family of cycles $(m_k)_{k \geq 1}$ which are conjectured to represent non-trivial classes $[m_k] \in H_{4k}(\mathrm{Out}(F_{2k+2});\QQ)$. Euler characteristic computations (see e.g.\ \cite{Borinsky_Vogtmann:2023} and references therein) show that there exist more classes, also in odd degrees (of which none is known to date).

\subsubsection{Forested graphs} \label{sss:forested_graphs}

There are two slightly different complexes of ``forested graphs" in the literature, both related to the rational (co-)homology of $\mathrm{Out}(F_n)$: The one introduced in \cite{convogt} (see also \cite{ConantVogtmann:Morita1,Vogtmann:Intro_gc}) is isomorphic to $\lie_N$ and computes thus the group \emph{co}homology of $\mathrm{Out}(F_n)$ (with a degree shift and the representation depending on $N$). The idea is to ``blow up" a $\mathrm{Lie}$-graph $G$ by inserting all the decorating trees $T_v$ at $v\in V_G$ and identifying their leaves with the corresponding half-edges of $G$. The result is a pair $(G',F)$ where $G'$ is the blown up graph and $F \subset G'$ is the forest that consists of all the inserted trees $T_v$. One checks that an odd orientation of the $\mathrm{Lie}$-graph $G$ is equivalent to ordering the edges of $F$. The forested graph complex is the quotient of the free vector space spanned by these forested graphs, subject to relations 
\begin{align}
  \label{eq:orient_pair_i} & (G,F,-\eta_F)= -(G,F,\eta_F), \tag{$i'$} \\ 
 \label{eq:orient_pair_ii} & (G,F,\eta_F)= (G',F', \varphi^* \eta_F) \text{ for isomorphisms } \varphi \colon G' \overset{\sim}{\to} G \text{ with } \varphi(F')=F, \tag{$ii'$}
\end{align}
plus a variant of the \hyperref[IHX]{IHX identity}. The differential adds edges to $F$ in all possible ways such that $F\cup e$ is still a forest in $G'$. 

The same construction works for the even case: The above described complex of forested graphs $(G,F)$ with \textit{odd} orientations, that is, using elements of $\det \QQ^{E_F}\otimes \det \homgra$, is isomorphic to $\lie_N$ for $N$ \textit{even}\footnote{This opposition of parity might suprise at first, but arises from the simple fact that for any finite dimensional $k$-vector space $V$ we have $\det V\otimes \det V\cong \det V\otimes (\det V)^\vee \cong k$ (because $\det V$ is 1-dimensional).}.

\begin{rem}
    In \cite{convogt} it is shown how also the (odd) associative graph complex relates to a modification of the above forested graph complex (forgetting the IHX-relation and restricting to the ``surface subcomplexes" described at the end of \cref{ss:gc_ass}).
\end{rem}

The second variant of a complex of forested graphs is slightly simpler. It was introduced in \cite{ConantVogtmann:Morita2} (based on \cite{cv,hv}). In geometric terms it arises from a cubical subdivision of the moduli space of graphs (\cref{ss:cubical_decomposition}) which is a rational classifying space for $\mathrm{Out}(F_n)$. It computes thus the rational \emph{homology} of $\mathrm{Out}(F_n)$ (with coefficients depending on the parity of $N$). From now on we work with this \textbf{forested graph complex} and denote it by $\gf_N$.

To define it we consider pairs $(G,F)$ where $G$ is a finite connected graph with all vertices at least trivalent and $F\subset G$ a forest in $G$. Here it will be convenient to think of a subgraph $\gamma \subset G$ as a (possibly disconnected) graph $(V(\gamma),E(\gamma))$ with $V(\gamma)=V(G)$ and $E(\gamma)\subset E(G)$. This means the forest $F$ is specified by a subset of edges in $E_G$; it automatically contains all the vertices of $G$. We define the \textbf{degree} of a pair $(G,F)$ by
 \[
 \deg (G,F) =e_F -Nh_G.
 \]

Let $\gf_N$ be the rational vector space spanned by linear combinations of triples $(G,F,\eta_F)$ where $F$ is a forest in $G$, and $\eta_F \in \det \QQ^{E_F}\otimes \det \homgra$ an orientation, subject to the relations \eqref{eq:orient_pair_i} and \eqref{eq:orient_pair_ii}.
Define a map $D \colon \gf_N \to \gf_N$ by $D=d-\delta$ with
\be \label{eq:cube_diffential}
\begin{split}
d[G,F,\eta_F] & = \sum_{e \in E(F)} [G\sla e,F\sla e,\eta_{F\sla e}] , \\
\delta[G,F,\eta_F] & = \sum_{e \in E(F)} [G,F\setminus e,\eta_{F \setminus e}],
\end{split}
\ee
where $\eta_{F\sla e}=\eta_{F \setminus e}$ are defined as above in \cref{eq:orient_collapse_even}. 

Let $\gf_N^{(g)}$ denote the subcomplex spanned by elements $[G,F,\eta_F]$ with $h_G=g$. We will see in \cref{s:MGandGC} that 
\[
H_\bullet(\gf_N^{(g)}) \cong H_{\bullet+Ng}(\mathrm{Out}(F_g),\QQ_N), \quad \QQ_N=\begin{cases}
    \QQ &  N \text{ even,} \\
    \det & N \text{ odd.}
\end{cases}
\]
Note the convention on $\QQ_N$ is opposite to the one in \cref{thm:Lie_and_Out}. Dualising and using \cref{thm:Lie_and_Out} we find\footnote{A proof on the level of (co)chains, using the other forested graph complex mentioned in the beginning of this section, for $N$ odd, can be found in \cite{Vogtmann:Intro_gc}.} the following relation. 
\begin{cor} \label{cor:for_gc_and_lie_gc}
The (dual of) forested graph homology is isomorphic to Lie graph homology, $H_\bullet(\gf_N^{(g)})^\vee \cong   H^{\bullet+Ng}(\mathrm{Out}(F_g),\QQ_N) \cong H_{(3-2N)g-3-\bullet}(\lie_N^{(g)})$. 
\end{cor}

Lie graph homology is thus related to the cohomology of $\mathrm{Out}(F_g)$ and thereby also to the homology of its rational classifying space, a moduli space of (metric) graphs. 
This marks the first appearance of our geometric avatar. Below we will see that also the commutative graph complex can be realised using this moduli space. 
In order to make this precise and further investigate this connection, we now turn our attention towards spaces of metric graphs.

\section{Moduli spaces of graphs} \label{s:moduli_space_of_graphs}

In this section we define various moduli spaces of graphs and introduce some subspaces that will become useful later in \cref{s:MGandGC}.

\subsection{A space of graphs}\label{ss:space_of_graph}
The moduli space of graphs has a convenient definition as realization of a category of graphs (e.g.\ \cite{cgp}), or as quotient of Outer space by the action of $\mathrm{Out}(F_g)$ \cite{cv} (see \cref{rem:outer_space}). Here we follow the construction from \cite{BMV:Tropical_Torelli,Chan:Combinatorics_of_trop_Torelli} of the moduli space of \textit{tropical curves} which are certain weighted metric graphs: A \textbf{weighting} on $G$ is a map $w \colon V_G \to \NN$. The \textbf{genus} of a weighted graph is 
\[
g(G,w)= h_G + \sum_{v \in V_G} w(v) .
\]

Let $g\geq 2$. Define a category $\m G_g$ by

\begin{itemize}
    \item $\mathrm{ob}(\m G_g)$ is the set of isomorphism classes of weighted graphs of genus $g$ and for every $v\in V_G$ we have $2 w(v) + \val v > 2$. We henceforth call such graphs \textbf{stable}. Note that, if $w\equiv 0$, then this is the same as requiring every vertex to be at least trivalent. As a consequence, $\m G_g$ is finite.
    \item $\mathrm{hom}(\m G_g)$ consists of subgraph collapses or isomorphisms of weighted graphs. Collapsing a subgraph $\gamma \subset G$ has the following effect on the weighting: If $\gamma$ collapses to a vertex $v\in V_{G/\gamma}$, then its new weight is
    \[
    w(v)=h_\gamma + \sum_{ v'\in V_\gamma} w(v').
    \]
\end{itemize}

For each $G \in \m G_g$\footnote{We henceforth abbreviate $[(G,w)]$ by $G$ to keep the notation light.} let 
\[
\sigma(G) = \mb R_{\geq 0}^{E_G}= \set{x \colon E_G \to \mb R_{\geq 0} }       
\]
denote the space of metrics $x$ on $G$. Given a morphism $f \in \mathrm{hom}(\m G_g)$, define a map $\sigma(f) \colon \sigma\left(f(G)\right) \to \sigma(G) $ by
\[
x' \longmapsto x \ \text{ where }  \ x_e= \begin{cases} x'_{e'} & \text{ if } e=f(e'), \\
                             0 & \text{ if $f$ collapses $e$.} \end{cases}
\]
This defines a functor from $\m G_g$ to the category of topological spaces. The \textbf{moduli space of weighted genus $g$ metric graphs} ${\m {MG}_g}$ is defined as the colimit of this functor. 

Note that on each cell $\sigma(G)$ we have a (linear) group action of $\Aut(G)\subset \hom (\m G_g)$. As a consequence $\m {MG}_g$ is an orbifold. In fact, it is a \textit{generalized cone complex} \cite{BBCMMW:gen_cone_cplx}, a union of cones 
\[
\ti \sigma(G)= \sigma(G) / \mathrm{Aut}(G) , 
\]
glued together by identifying $\ti \sigma(G/\gamma)$ with $\set{x_e=0 \mid e\in E_\gamma} \subset  \ti \sigma(G)$. For instance, if $G$ is the theta graph (\cref{fig:theta}), then the corresponding cone in $\m {MG}_2$ is the quotient of $\mb R_{\geq 0}^3$ by the action of $S_3$ permuting the coordinates $x_1,x_2,x_3$ (see \cref{fig:cells_and_face_relation}). 

Alternatively, one can start with the collection of cones $\sigma(G)$, glue them together along common boundaries $\sigma(G/\gamma)\cong\set{x_e=0 \mid e\in E_\gamma} \subset \sigma(G)$, and then take the quotient with respect to the equivalence relation 
\[
x \sim x' \quad  \Longleftrightarrow \quad  \exists \ \varphi \in \Aut(G) : x=x'\circ \varphi.  
\]
This point of view has the advantage that we can define objects on $\m {MG}_g$ by using the ``intermediate complex" formed by the cells $\sigma(G)$ and requiring each piece of of data to be equivariant with respect to the ``local action" of $\Aut(G)$. This is for instance done in \cite{brown21} to set up a de Rham theory on a certain subspace $MG_g$ of $\m {MG}_g$, or in \cite{mb2,mb-dk} to study Feynman integrals on this space.

It is customary to normalize the metrics on graphs ($\m {MG}_g$ is contractible, thus quite boring). For this define 
\[
MG_g=  \m {MG}_g / \RR 
\]
where $\RR$ acts by rescaling the metrics. In other words, $MG_g$ is the link of the vertex $*_g$ that represents the graph with no edges and a single vertex of weight $g$. 

To identify $MG_g$ with a subset of $\m {MG}_g$ define a map $\mathrm{vol} \colon \m {MG}_g \to \RR_{\geq 
0}$ that measures the \emph{volume} of a metric graph, $\mathrm{vol}(x)= \sum_{e\in E_G} x_e$. Then $MG_g \cong  \inv{ \mathrm{vol} } (1)$. 
This turns the cones $\sigma(G)$ into simplices
\[
\Delta_G = \set{ \ [x] : x \in \sigma(G) } \subset \PP^G=\PP(\RR^{E_G}). 
\] 
From now on we work with this \textbf{moduli space of normalized weighted genus $g$ metric graphs}, and henceforth refer to $MG_g$ simply as the \emph{moduli space of graphs}.

\begin{figure}[h]
 \begin{tikzpicture}[scale=1]
 \coordinate (b) at (0,-1.8);
 \filldraw[white] (b) circle (0.01);
\coordinate (l) at (-1,0);
  \coordinate  (r) at (1,1); 
  \coordinate  (o) at (0,1.7);
  \coordinate (c) at (0,0); 
      \coordinate (h) at (1,2);
     \coordinate (hd) at (-.5,1.5);
      \coordinate (s1) at (0.6,0.6); 
       \coordinate (s2) at (0.4,.8); 
        \coordinate (s3) at (-.25,.75); 
  \fill[gray!23] (c) -- (hd) -- (h) -- (r) -- (c);
  \fill[red!18] (s1) -- (s2) -- (s3);
  \draw[thin,red!40] (s1) -- (s2) --(s3) -- (s1);
     \draw[dashed] (c) -- (h);
     \draw[dashed] (c) -- (hd) node[left] {$\scriptstyle{x_1=x_2=x_3}$};
     \draw[dashed, black!50] (hd) -- (r);
   \draw (c) -- (l) node[left] {$x_2$};
   \draw (c) -- (r) node[right] {$x_1$};
   \draw (c) -- (o) node[above] {$x_3$};
  %\filldraw[fill=gray!30] (c) circle (0.0666);
  %%% graph
  \coordinate (v1) at (-.5,-1);
  \coordinate (v2) at (.5,-1);
   \draw (v1) to[out=90,in=90] node[above,yshift=-.05cm] {$\scriptstyle{1}$} (v2);
   \draw (v1) -- node[above,yshift=-.1cm] {$\scriptstyle{2}$} (v2);
   \draw (v1) to[out=-90,in=-90] node[above,yshift=-.1cm] {$\scriptstyle{3}$} (v2);
  \filldraw[fill=black] (v1) circle (0.06);
  \filldraw[fill=black] (v2) circle (0.06);
  \end{tikzpicture} 
  %%%%%%%%%%%%%%%
\qquad \qquad
  %%%%%%%%%%%%%%%
   \begin{tikzpicture}[scale=1]
  \coordinate  (r) at (1,1); 
   \coordinate  (o) at (0,1.6);
  \coordinate (c) at (0,0); 
  \coordinate (h) at (1,2);
   \coordinate (s1) at (0.6,0.6); 
   \coordinate (s2) at (0.4,.8); 
  \fill[gray!23] (c) -- (r) -- (h) -- (c);
   \draw[red!33,thick] (s1) -- (s2);
   \draw[dashed] (c) -- (h);
   \draw (c) -- (r) node[right] {$x_1$};
   \draw (c) -- (o) node[above] {$x_3$};
      %%% graph
  \coordinate (v1) at (0,-1);
  \draw[scale=2.5] (v1) to[in=-140,out=130,loop] node[midway,xshift=-.1cm] {$\scriptstyle{1}$} (v1);
     \draw[scale=2.5] (v1) to[in=50,out=-40,loop] node[midway,xshift=.1cm] {$\scriptstyle{3}$} (v1);
  \filldraw[fill=black] (v1) circle (0.06);
  \end{tikzpicture} 
    %%%%%%%%%%%%%%%
\qquad \qquad
  %%%%%%%%%%%%%%%
   \begin{tikzpicture}[scale=1]
\coordinate (rr) at (1,0);
  \coordinate  (r) at (1,1); 
   \coordinate  (o) at (0,1.6);
  \coordinate (c) at (0,0); 
    \coordinate (h) at (1,2);
     \coordinate (ho) at (1.95,2);
      \coordinate (hu) at (1.95,1);
       \coordinate (s1) at (0.6,0.6); 
       \coordinate (s2) at (0.4,.8); 
        \coordinate (s3) at (.666,0); 
  \fill[gray!23] (c) -- (rr) -- (hu) -- (ho) -- (h) -- (c);
   \fill[red!18] (s1) -- (s2) -- (s3);
   \draw[thin,red!40] (s1) -- (s2) --(s3) -- (s1);
     \draw[dashed] (c) -- (h);
     \draw[thin,dashed,black!66] (rr) -- (ho);
   \draw (c) -- (rr) node[right] {$x_2$};
   \draw (c) -- (r) node[above right] {$x_1$};
   \draw (c) -- (o) node[above] {$x_3$};  %\filldraw[fill=gray!30] (c) circle (0.0666);
   %%% graph
  \coordinate (v1) at (-.1,-1);
  \coordinate (v2) at (.4,-1);
   \draw[scale=2.5] (v1) to[in=-140,out=130,loop] node[midway,xshift=-.1cm] {$\scriptstyle{1}$} (v1);
   \draw (v1) -- node[above,yshift=-.1cm] {$\scriptstyle{2}$} (v2);
   \draw[scale=2.5] (v2) to[in=50,out=-40,loop] node[midway,xshift=.1cm] {$\scriptstyle{3}$} (v2);
  \filldraw[fill=black] (v1) circle (0.06);
  \filldraw[fill=black] (v2) circle (0.06);
  \end{tikzpicture} 
   \caption{Two neighboring cones $\sigma(G)$ and their common face. The quotient spaces $\ti \sigma(G)= \sigma(G) / \Aut(G)$ are colored in grey, the cells $\Delta_G$ in red. Note that all cones in $\m{MG}_g$ contain the origin. It corresponds to the (metric) graph with no edges and a single vertex of weight $g$.}\label{fig:cells_and_face_relation}
\end{figure}
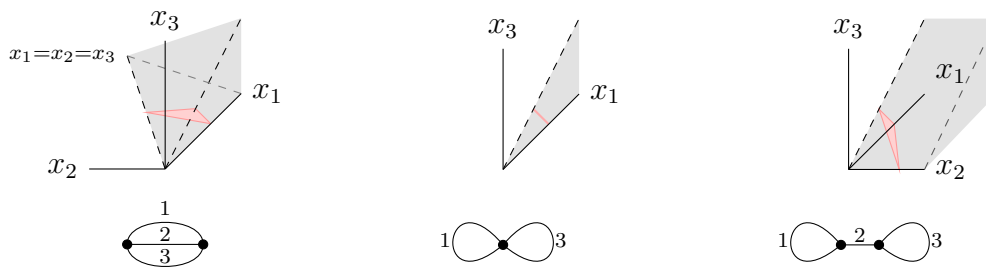

\begin{prop} \label{prop:dim_mod_space}
    The dimension of $MG_g$ is $3g-4$. 
\end{prop}

\begin{proof}
  Recall the stability condition: $2w(v)+\val{v} >2$ must hold for every vertex $v\in V_G$. Suppose first that $G$ is of weight zero. From the Euler characteristic of $G$ we get $e_G=v_G+h_G-1$, hence $e_G$ is maximal if all vertices are trivalent. For such a graph we have $2e_G=3v_G$ and therefore $e_G=3h_G-3$. This can not be improved by adding weights: If we attach an ``antenna" whose end has weight 1 to an edge of a graph $G'$ with $h_{G'}=g-1$, the result $G''$ has $e_{G''}=e_{G'}+2=3(g-1)-1<3g-3$. Therefore, $\max\set{e_G \mid G \in \mathrm{ob}(\m G_g) }=3g-3=\dim \m{MG}_g $, hence $\dim MG_g=3g-4$.
\end{proof}

\begin{rem}
    All of the above can be done for graphs with legs \cite{chkv,cgp2} or with colored edges \cite{mb-mm}. The only essential difference is that graph isomorphisms must respect the additional leg or color structure, everything else applies verbatim (only internal edges have lengths). The same holds for the case of ribbon graphs; see \cref{ss:ass_cellular}. 
\end{rem}

 The notion of building a space out of quotient cells can be formalized in many different ways. In \cite{cgp} this is done using the language of \textit{symmetric semi-simplicial complexes}. More generally, these are instances of (links of) \textit{generalized cone complexes} or \textit{symmetric CW-complexes}; see \cite[\S 2]{BBCMMW:gen_cone_cplx} and references therein. We return to this point in \cref{s:MGandGC}.

\subsection{Subspaces of $MG_g$}

We introduce some subspaces of $MG_g$ that will be relevant in the following sections. Define
\begin{itemize}
    \item $MG_g^k$ as the subspace of (metrics on) graphs with total weight $k$,
    \item $MG_g^{\geq k}$ as the \emph{subcomplex} of graphs with total weight at least $k$.
\end{itemize}

Most important for us is the case $k=0$. Since collapsing a tadpole edge increases the total weight of a graph by 1, $MG_g^0$ is not a subcomplex of $MG_g$. It is (the quotient of) a cell complex with some of its cells deleted: If $G$ has no weights, then a face 
\[
\Delta_{G/\gamma}\cong\set{x_e=0 \mid e \in E_\gamma\subset E_G}\subset \Delta_G
\]
 lies in $MG_g^0$ if and only if $h_\gamma=0$, that is, if $\gamma$ is a forest in $G$. 
If $h_\gamma>0$, then $\Delta_{G/\gamma}$ lies in the complement $MG_g\setminus MG_g^0=MG_g^{\geq1}$. 

\begin{rem} \label{rem:outer_space}
  Classically, $MG_g^0$ arises as the quotient of Outer space $O_g$ by the group action of $\mathrm{Out}(F_g)$ \cite{cv}. Here $O_g$ is a moduli space of \textit{marked} metric graphs where a marking is a homotopy equivalence $m\colon G \to R_g$ with $R_g$ a vertex with $g$ tadpoles attached. Two markings $m\colon G \to R_g$ and $m'\colon G' \to R_g$ represent the same element of $O_g$ if there is an isomorphism $\varphi\colon G \to G'$ such that $m$ is homotopic to $m'\circ \varphi$. Since $\pi_1(R_g)$ is isomorphic to the free group $F_g$, the group $\mathrm{Out}(F_g)$ of outer automorphisms of $F_g$ acts on $O_g$ by changing the markings. The quotient space is $MG_g^0$. In other words, Outer space is the total space of a fibration over $MG_g^0$, where the fiber over $x\in \Delta_G$ consists of all (homotopy classes of) markings of $G$ (with respect to the metric $x$).  
\end{rem}

Since $MG_g^0$ is not a complex, it is often useful to replace it by better behaved spaces. Firstly, there is the obvious (semi-simplicial) completion, obtained by adding all the missing faces; this is just $MG_g$. Secondly, we can truncate small ``neigborhoods of infinity" in each cell $\Delta_G$ to obtain polytopes $J_G$ that assemble to a compact space $\ti MG^0_g$, homotopy equivalent to $MG_g^0$. This is equivalent (homeomorphic) to blowing-up $MG^0_g$ along its faces at infinity in increasing order of dimension. For the details of this construction we refer to \cite{v-bord}; see also \cite{mb2,brown21,brown:bordifications}.  

One can also retract $MG_g^0$ onto its \textbf{spine}, a subcomplex $S_g$ of the barycentric subdivision of $MG_g$. As an abstract complex $S_g$ is defined as the order complex of the poset 
\[
\left( \mathrm{ob}(\m G_g), \leq \right), \quad G \leq G' \text{ if there is a forest } F \subset G' \text{ with } G=G'/F.
\footnote{Recall that we are working with isomorphism classes. This relation is therefore required to hold only up to isomorphism.}
\]
 The geometric realization of this complex can be embedded in $MG_g$; see \cref{fig:sunrise_cell,fig:cubes_in_simplex}, and \cref{ss:HE_MGandSpine}. 
 Then there is a deformation retract $r: MG_g^0 \to S_g$ defined by collapsing all the cells that have vertices in $MG_g^{\geq 1}$ onto their faces in $S_g$ \cite{cv}. This spine of $MG_g^0$ is in fact a cube complex as we will explain in the next section.

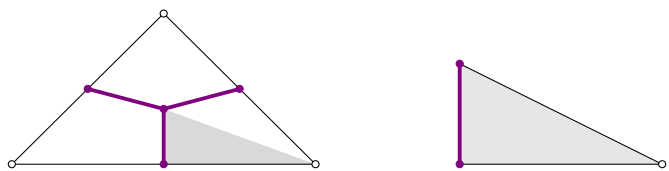
\begin{figure}[h]
 \begin{tikzpicture}[scale=.666]
\coordinate (l) at (-3,0);
  \coordinate  (r) at (3,0); 
   \coordinate  (o) at (0,3);
  \coordinate (lo) at (-1.5,1.5); 
   \coordinate (lr) at (0,0); 
  \coordinate (ro) at (1.5,1.5); 
  \coordinate (c) at (0,1.1); 
  \fill[gray!30] (c) -- (r) -- (lr) -- (c);
   \draw (l) -- (r);
   \draw (l) -- (o) ;
   \draw (r) -- (o) ;
     \draw[line width=0.5mm, violet] (lo) -- (c);
   \draw[line width=0.5mm, violet] (lr) -- (c);
   \draw[line width=0.5mm, violet] (ro) -- (c);
  \filldraw[violet] (lo) circle (0.07);
  \filldraw[violet] (lr) circle (0.07);
 \filldraw[violet] (ro) circle (0.07);
  \filldraw[violet] (c) circle (0.07);
 \filldraw[fill=white] (l) circle (0.07);
 \filldraw[fill=white] (r) circle (0.07);
 \filldraw[fill=white] (o) circle (0.07);
  \end{tikzpicture} 
  %%%%%%%%%%%%%%%
\qquad \qquad
  %%%%%%%%%%%%%%%
   \begin{tikzpicture}[scale=.666]
\coordinate (l) at (-2,0);
  \coordinate  (r) at (2,0); 
   \coordinate  (o) at (-2,2);
   \fill[gray!20] (l) -- (o) -- (r) -- (l);
   \draw (l) -- (r);
   \draw (l) -- (o) ;
   \draw (r) -- (o) ;
     \draw[line width=0.5mm, violet] (l) -- (o);
  \filldraw[violet] (l) circle (0.07);
  \filldraw[violet] (o) circle (0.07);
 \filldraw[fill=white] (r) circle (0.07);
  \end{tikzpicture} 
   \caption{On the left the simplex $\Delta_{B_3}$ and its spine for the 3-banana $B_3$ (aka sunrise or theta graph). On the right the quotient cell $\ti \Delta_{B_3}= \Delta_{B_3} / S_3$ and its spine.}\label{fig:sunrise_cell}
\end{figure}

\begin{cor} \label{cor:dim_subspaces}
    The dimension of $MG_g^0$ is $3g-4$, the dimension of $MG_g^{\geq1}$ is $3g-5$, and the dimension of the spine $S_g$ is $2g-3$.
\end{cor}
    
\begin{proof}
    All statements follow from the model constructed in the proof of \cref{prop:dim_mod_space}. 
\end{proof}

\subsection{Subdivision into cubes}\label{ss:cubical_decomposition}

The moduli space of graphs can also be described as a union of (quotient) cubes. For this let $\beta(\Delta_G)$ denote the barycentric subdivision of a simplex $\Delta_G$. Its vertices are the points
\[
x \in \Delta_G : x_e=0  \text{ for all }e \in E_\gamma,  \ x_i=x_j  \text{ for all } i,j\notin E_\gamma,
\]
that is, the barycenters of $\Delta_{G/\gamma}$ for subgraphs $\gamma\subset G$ with at most $e_G-1$ edges. A simplex of $\beta (\Delta_G)$ can thus be described by a sequence of edge-collapses 
\[
 G' \longrightarrow G'/e_1 \longrightarrow G'/\set{e_1 \cup e_2} \longrightarrow \cdots \longrightarrow G'/\set{e_1 \cup \ldots \cup e_k}
\]
where $G'$ denotes a cograph $G/ \gamma$ with $e_\gamma < e_G$ and $\set{e_1,\ldots,e_k}\subset E_{G / \gamma}$. Since any order on $I=\set{e_1,\ldots,e_k}$ produces such a simplex, they can all be grouped together to form a $k$-dimensional cube. It follows that the interior $\overset{\circ}{\Delta}_G$ decomposes into a disjoint union of open cubes indexed by subgraphs $\gamma \subset G$ with $ e_\gamma < e_G$, and likewise for the boundary $\partial \Delta_G$.

Let $\cube \subset \Delta_G$ denote the cube 
\[
\cube = \set{ [x] : x_i=x_j \text{ for all }i,j \notin E_\gamma} \cong [0,1]^{e_\gamma}.
\]
Its facets are indexed by the edges $e \in E_\gamma$ and come in two types, either $\set{x_e=0}$, or $\set{x_e=x_i \mid i \notin E_\gamma}$. The first type is isomorphic to the cube $c(G/e,\gamma/e)$, the second to $c(G,\gamma \setminus e)$.

Note that an isomorphism $\varphi\colon G \overset{\sim}{\to} G'$ induces an isomorphism 
$
c(\varphi) \colon \cube \overset{\sim}{\to} c(G',\varphi(\gamma))$. Define
\[
\Aut(G,\gamma)= \set{ \varphi \in \Aut(G) \mid \varphi(E_\gamma)=E_\gamma }
\]
as the group of automorphisms of the pair $(G,\gamma)$ and denote by $\ti{c}(G,\gamma)$ the quotient $\cube/\Aut(G,\gamma)$.

Let $\m P_g$ be the category with 
\begin{itemize}
    \item $\mathrm{ob}(\m P_g)$ the set of isomorphism classes of pairs $(G,\gamma)$ where $G$ is stable, has genus $g$ and $\gamma \subset G$.
    \item $\mathrm{hom}(\m P_g)$ consists of isomorphisms $(G,\gamma)\mapsto (G',\gamma')$, edge collapses $(G,\gamma)\mapsto (G/e,\gamma/e)$, and edge deletions $(G,\gamma) \mapsto (G,\gamma-e)$.
\end{itemize}

Putting all of the above together it follows that the moduli space of graphs can be described as a union of quotient cubes 
\be \label{eq:MG_cube_decomposition}
MG_g = \Big( \bigcup_{(G,\gamma)\in \mathrm{ob}(\m P_g)} \ti{c}(G,\gamma)  \Big)/\sim 
\ee
with $\sim$ denoting the face relation that identifies for all $\eta \subset \gamma$
\begin{align*}
    \ti{c}(G/\eta,\gamma/\eta) & \ \text{ with } \ \set{x_e=0 \mid e\in E_\eta} \subset  \ti{c}(G,\gamma), \\
    \ti{c}(G,\gamma \setminus \eta) & \ \text{ with } \ \set{x_i=x_j \mid i\in E_\eta, j\notin E_\gamma} \subset  \ti{c}(G,\gamma).
\end{align*}

The discussion in \cref{ss:space_of_graph} applies verbatim; we can construct $MG_g$
by first taking the union of cubes $\cube$, glued together along faces 
$c(G/\eta,\gamma/\eta)$ and $c(G,\gamma \setminus \eta)$, and then take the quotient with respect to the ``local" action by the groups $\Aut(G,\gamma)$.

The spine $S_g$ is (homeomorphic to) the cubical subcomplex of pairs $(G,\gamma)$ where $\gamma$ is a forest in $G$.

\begin{example}
    Consider $\Delta_{B_3}$, depicted in \cref{fig:sunrise_cell}. It decomposes into 
    \begin{itemize}
        \item seven 0-dimensional cubes: $c(B_3,\varnothing)$ (center), $ c(B_3/e_i,\varnothing)$ (midpoints of facets), $ c(B_3/\set{e_i,e_j},\varnothing)$ (corners).
        \item nine 1-cubes: $ c(B_3,e_i)$ (violet) and $c(B_3/\set{e_i,e_j})$ (boundary).
        \item three 2-cubes: $c(B_3,\set{e_i,e_j})$.
    \end{itemize} 
    On the right in \cref{fig:sunrise_cell} is the quotient $\ti \Delta_{B_3}=\ti c(B_3,\set{e_i,e_j})$, a single 2-cube that is ``folded" along its  diagonal (because $\Aut(B_3,\set{e_i,e_j})$ acts by permuting the coordinates $x_i$ and $x_j$).
\end{example}

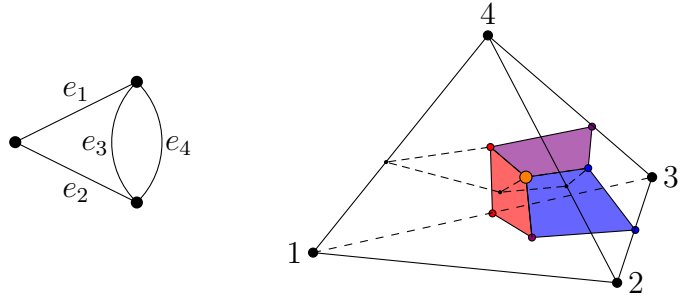
\begin{figure}[h]
  \begin{tikzpicture}[scale=.8]
  \coordinate (h) at (0,0) node {};
   \coordinate (v0) at (0,2.5);
   \coordinate  (v1) at (2,1.5);
   \coordinate (v2) at (2,3.5);
   \draw[] (v0) -- (v1) node [midway,below]{$\color{black}e_2$};
   \draw (v0) -- (v2) node [midway,above]{$e_1$};
   \draw (v2) to[out=-135,in=135] (v1) node [xshift=-.55cm,yshift=.75cm]{$e_3$};
   \draw[] (v2) to[out=-45,in=45] (v1) node [xshift=.55cm,yshift=.75cm]{$\color{black}e_4$}; 
   \fill[black] (v0) circle (.1cm);
   \fill[black] (v1) circle (.1cm); 
   \fill[black] (v2) circle (.1cm); 
  \end{tikzpicture}
  %%%%%%%%%%%%%%%%%%%%%%%%%%%%%%%%%%%%%%%%%%%%%%%%%%%
  \quad \quad 
  %%%%%%%%%%%%%%%%%%%%%%%%%%%%%%%%%%%%%%%%%%%%%%%%%%%%%%
\begin{tikzpicture}[scale=2]
\coordinate (v1) at (-1,0);
\coordinate  (v2) at (1,-.2); 
\coordinate  (v4) at (0.15,1.44); 
\coordinate  (v3) at (1.23,.5);
 \coordinate (c2) at (.44,.1); 
\coordinate (c3) at (0.1666,.7); 
  \coordinate (c4) at (0.18,0.26);   
 \coordinate (c1) at (.4,0.5);
  \coordinate (k3) at (1.12,0.15);
   \coordinate (k4) at (.81,0.56);
   \coordinate (l) at (.835,.835);
   \coordinate (s) at (.666,.44);
   \coordinate (sf) at (.23,.4);
   \coordinate (sb) at (-.52,.6);
         % CUBEs 
         \fill[fill=red!60] (c1) -- (c2) -- (c4) -- (c3) -- (c1);
         \fill[fill=blue!60] (c1) -- (c2) -- (k3) -- (k4) -- (c1);
         \fill[fill=violet!60] (c1) -- (c3) -- (l) -- (k4) -- (c1);
    \draw (c1) to (c2);
    \draw[] (c2) to (c4);
  \draw[] (c4) to (c3);
  \draw (c3) to (c1);
  \draw (k3) to (c2);
    \draw (k4) to (c1);
    \draw (k4) to (k3);
    \draw (k4) to (l);
    \draw (c3) to (l);
     \draw[densely dashed] (c1) -- (k4) -- (s) -- (sf) -- (c1);
      \draw[densely dashed]  (sf) -- (sb) -- (c3) -- (c1);
   \filldraw[fill=red!60!blue] (c2) circle (0.023);
   \filldraw[fill=red] (c3) circle (0.023);
     \filldraw[fill=red] (c4) circle (0.023); 
      \draw (c1) to (c2);
           \filldraw[fill=orange] (c1) circle (0.04);
               \filldraw[fill=blue] (k3) circle (0.023);
                \filldraw[fill=blue] (k4) circle (0.023);
                  \filldraw[fill=violet] (l) circle (0.023);
                      \filldraw[fill=black] (s) circle (0.01);
    \filldraw[fill=black] (sf) circle (0.01);
     \filldraw[fill=black] (sb) circle (0.01);
                 % SIMPLEX
 \filldraw[fill=black] (v1) circle (0.03) node[left]{$1$};
  \filldraw[fill=black] (v2) circle (0.03) node[right]{$2$};
   \filldraw[fill=black] (v4) circle (0.03) node[above]{$4$};
    \filldraw[fill=black] (v3) circle (0.03) node[right]{$3$};  
  \draw[thin] (v1) to (v2);
  \draw[thin] (v1) to (v4);
 \draw[thin] (v2) to (v4);
  \draw[thin] (v4) to (v3);
  \draw[thin,dashed] (v1) to (v3);
    \draw[thin] (v2) to (v3);
  \end{tikzpicture}  
\caption{A graph $G$ and three cubes $c(G,T)$ for $T=\set{e_2,e_4}$ (red), $T=\set{e_1,e_2}$ (cyan), and $T=\set{e_1,e_4}$ (blue) inside the simplex $\Delta_G$; the orange vertex lies in the center of $\Delta_G$. The cubes associated to the other two spanning trees of $G$ are indicated by dashed lines. Together the five cubes form the spine of $\Delta_G$.}\label{fig:cubes_in_simplex}
\end{figure}

\subsection{The homotopy equivalence $MG_g^0 \simeq S_g$.}\label{ss:HE_MGandSpine}
The fact that $S_g$ is a deformation retract of $MG_g^0$ was first proven in the context of Outer space \cite{cv}. Homotopy equivalence also follows from the existence of a fibration $MG_g^0 \to S_g$ \cite{mb:ltd}. 

We give here an explicit proof of this fact. Restricting the decomposition \eqref{eq:MG_cube_decomposition} to pairs $(G,\gamma)$ with $h_\gamma=0$ gives rise to a decomposition of the spine
\[
S_g =  \Big( \bigcup_{(G,F)} \ti{c}(G,F)  \Big)/\sim ,
\]
where the union is over all (isomorphism classes of) pairs $(G,F)$ with $F$ a forest in $G$. We write $S_G=\ti \Delta_G \cap S_g$ for the spine of a single cell.

\begin{thm} \label{thm:spine_retract}
The spine $S_g$ is a deformation retract of $MG_g^0$.
\end{thm}

\begin{proof}
We construct a deformation retract $r \colon MG_g^0 \to S_g$ first cell-by-cell, then show that each map is $\Aut(G)$-equivariant and respects face relations.

Consider a $n$-cube $\cube \subset \Delta_G$ with $n=e_\gamma$. It decomposes into $n!$ semi-open $n$-simplices $s(\sigma)$, indexed by $\sigma \in S_n$,
\[
s(\sigma)=s(G,\gamma,\sigma) = \left\{ [x_e]_{e \in E_G} \ \middle\vert \begin{array}{l} x_e>0 \quad \forall e \in E_G, \\
    x_i = x_j \quad \forall i,j \notin E_\gamma ,\\
     x_k \leq x_\ell  \quad \forall k\in E_\gamma, \ell \notin E_\gamma, \\
     x_{\sigma(a)} \leq x_{\sigma(b)} \quad \forall a,b \in E_\gamma 
  \end{array}\right\}.
\]
In other words, the closures of the $s(G,\gamma,\sigma)$ for $\gamma\subset G$ and $\sigma \in S_{e_\gamma}$ form the barycentric subdivsion $\beta(\Delta_G)$ of $\Delta_G$.

For each $\sigma$ we define a forest $F(\sigma)=F(G,\gamma,\sigma) \subset \gamma$ by 
\[
E_{F(\sigma)}= \max \Big\{ \set{e_{\sigma(1)}, \ldots, e_{\sigma(k)}} \text{ is a forest in }\gamma \Big\}. 
\]
In words, $\sigma$ specifies an order on $E_\gamma$, and we take the largest ordered subset such that it includes $e_{\sigma(1)}$ and is a forest in $\gamma$.

Define a map $ r_\sigma= r_{G,\gamma,\sigma} \colon s(\sigma)  \to c\big(G,F(\sigma)\big)$ by 
\[
    r_\sigma(x_i)=
    \begin{cases}
    x_{\sigma(1)} & \text{ if } E_{F(\sigma)} = \varnothing, \\
    x_i & \text{ if } i=\sigma(j) \text{ for some } j\in \set{1,\ldots,k}, \\
    x_{\sigma(k+1)} & \text{ else. }
        \end{cases}
\]
See \cref{fig:retract} for an example.

\begin{figure}[h]
\begin{tikzpicture}[scale=.8]
\coordinate (v0) at (0,0);
  \coordinate  (v1) at (0,2.5);
   \coordinate  (v2) at (0,1.5);
   \draw[white] (v0) to (v2);
   \draw (v1) to[out=-30,in=30] node[midway,xshift=.25cm,yshift=-.16cm] {$e_2$} (v2);
   \draw (v1) to[out=-150,in=150] node[left,xshift=-.3cm,yshift=.33cm] {$G=$} node[midway,xshift=-.2cm,yshift=-.15cm] {$e_1$} (v2);
   \draw[scale=2.5] (v1) to[in=50,out=130,loop] node[midway,xshift=0cm,yshift=.18cm] {$e_3$} (v1);
  \filldraw[fill=black] (v1) circle (0.07);
  \filldraw[fill=black] (v2) circle (0.07);
  \end{tikzpicture}
  %%%%%%%%%%%%%%%%%%%%%%%%%%%%%%%%%%%%%%%%%%%%%%
  \quad  \quad 
  %%%%%%%%%%%%%%%%%%%%%%%%%%%%%%%%%%%%%%%%%%%%%%5
 \begin{tikzpicture}[scale=.8]
\coordinate (l) at (-3,0);
  \coordinate  (r) at (3,0); 
   \coordinate  (o) at (0,3);
  \coordinate (lo) at (-1.5,1.5); 
   \coordinate (lr) at (0,0); 
  \coordinate (ro) at (1.5,1.5); 
  \coordinate (c) at (0,1.1); 
   \draw[dashed] (l) -- (r);
   \draw (l) -- (o) ;
   \draw (r) -- (o) ;
   \fill[blue!50] (lo) -- (o) -- (c) -- (lo);
   \fill[blue!50] (lo) -- (l) -- (c) -- (lo);
   \fill[cyan!50] (ro) -- (o) -- (c) -- (ro);
   \fill[cyan!50] (ro) -- (r) -- (c) -- (ro);
   \fill[teal!50] (l) -- (r) -- (c) -- (l);
     \draw[line width=0.5mm, blue] (lo) -- (c);
   \draw[line width=0.5mm, cyan] (ro) -- (c);
  \filldraw[fill=black] (lo) circle (0.07);
 \filldraw[fill=black] (ro) circle (0.07);
  \filldraw[teal] (c) circle (0.08);
 \filldraw[fill=white] (l) circle (0.07) node[below]{$[1:0:0]$};
 \filldraw[fill=white] (r) circle (0.07) node[below]{$[0:1:0]$};
 \filldraw[fill=white] (o) circle (0.07) node[above]{$[0:0:1]$};
  \end{tikzpicture} 
   \caption{An illustration of the retraction $r\colon MG_g^0 \to S_g$ in a cell $\Delta_G$. The spine is the 1-dimensional subcomplex of $\beta(\Delta_G)$ with vertices $[1:0:1]$, $[1:1:1]$ and $[0:1:1]$, the colors indicate to which parts of it the other cells collapse. Note that the top vertex and the baseline are not part of $MG_g^0$.} \label{fig:retract}
\end{figure}
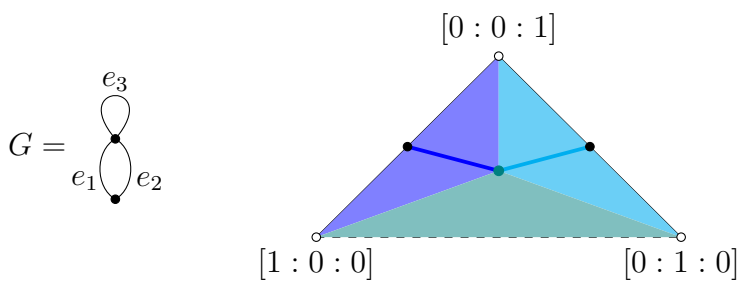

Let $\iota$ denote the inclusion $s(\sigma) \hookrightarrow \Delta_G$. Then $x_e \mapsto tx_e + (1-t) r_\sigma(x_e)$ defines a homotopy between $\iota|_{c(G,F(\sigma))} \circ r_\sigma$ and $\id_{s(\sigma)}$. 

In addition, each $r_\sigma$ is compatible with face relations:
\begin{enumerate}
    \item If $\sigma$ and $\tau$ differ by a transposition $(ij)$, then $s(\sigma) \cap s(\tau)\subset \set{x_i=x_j}$. On this locus the maps $r_\sigma$ and $r_\tau$ agree.
    \item The same holds for an intersection $s(G,\gamma,\sigma) \cap s(G,\gamma',\sigma')$ which lies in $c\big(G,F(G,\gamma,\sigma)\cap F(G,\gamma',\sigma')\big)$.
    \item Lastly, we consider $\set{x_e=0} \subset \overline{ s(\sigma) }$ for $e \in F(\sigma)$. This is homeomorphic to $s(G/e,\gamma/e,\sigma/e)$ where $\sigma/e$ denotes the induced order on $E_{\gamma/e}$. Moreover, $r_{G,\gamma,\sigma}$ extends to $\set{x_e=0}$ where it equals $f_e \circ r_{G/e,\gamma/e,\sigma/e}$ with $f_e$ denoting the face map $f_e \colon \PP^{G/e} \overset{\sim}{\to} \set{x_e=0} \subset \PP^G$.
\end{enumerate}

For the equivariance property note that $\varphi \in \Aut(G)$ maps $s(\sigma)\subset c(G,\gamma)$ bijectively to $s(\sigma')  \subset c\big(G,\varphi(\gamma)\big)$ with $\sigma'=\sigma\circ \varphi$. Furthermore, since $\varphi$ maps forests to forests, we have $F(\sigma')=\varphi\left( F(\sigma) \right)$. It follows that $r|_{\Delta_G}$ induces a well-defined map on the quotients $\ti \Delta_G \to  S_G$, and therefore the family $\set{r_\sigma}_{\sigma}$ assembles to a well-defined deformation retract $r \colon MG_g^0 \to S_g$.
\end{proof}

\section{Graph complexes as (relative) cellular chain complexes}\label{s:MGandGC}

In this section we relate the complexes $\comm_N$ and $\lie_N$ to certain cellular chain complexes associated to the moduli spaces introduced in \cref{s:moduli_space_of_graphs}. We finish with a short remark on the associative case. The discussion is split into subsections, according to the relevant graph complex.
\newline

The argument is based on one main feature of both the moduli space of graphs and its spine: Their ``natural" decomposition into cells indexed by graphs allows to describe them as a union of cones $CX$ where $X$ is either a simplex (moduli space of graphs) or a cube (spine), glued together along face relations. This decomposition gives rise to a small chain complex that is quasi-isomorphic to the respective complex of singular chains, and isomorphic to the relevant graph complex. 

A key tool for using this observation is the following
\begin{lem} \label{lem:quotient_homology}
    %Let $S^n$ denote the $n$-sphere, $D^n$ the $n$-disk. If $\Gamma$ is a finite linear group acting on $S^n$%by permuting the coordinates of $\RR^{n+1}$
    Let $\Gamma$ be a finite group acting on the $n$-sphere $S^n$. Then 
    \[
    \ti H_\bullet (  S^n/\Gamma ; \QQ ) = \begin{cases}
        \ti H_\bullet(S^n ; \QQ) & \text{action induces no orient.-rev.\ homeomorphisms,} \\
        0 & \text{else.}
    \end{cases}
    \]
\end{lem}

\begin{proof}
We use the following statement (see e.g.\ \cite{bredon:intro_compact_tragrps}): If $\Gamma$ is a finite group acting on a CW-complex $X$, then $H_\bullet(X; \QQ)$ is a $\Gamma$-module, and the space of its coinvariants 
\[
H_\bullet( X ; \QQ)_\Gamma=H_\bullet(X; \QQ)/ I, \quad  I=\langle g_*h -h \mid h\in H_\bullet(X; \QQ), g \in \Gamma \rangle,
\]
is isomorphic to the homology of the quotient, $H_\bullet( X ; \QQ)_\Gamma \cong
H_\bullet(X/\Gamma;\QQ) $.

Now consider $H_n(S^n)=\langle [S^n] \rangle$. If there exist $g\in \Gamma$ such that $g_*S^n=-S^n$, then $I=\langle 2[S^n] \rangle $ and therefore $H_n(S^n ; \QQ)_\Gamma=0$. Otherwise $I=0$, so that $H_n( S^n ; \QQ)_\Gamma=H_n( S^n ; \QQ)$.
\end{proof}

\subsection{$\comm_N$ for $N$ even} \label{ss:comm_cellular}

Recall from \cref{eq:d_in_comm} the definition of the differential $d$ on $\comm_N$, and suppose first that $G$ is a graph without multi-edges and tadpoles. From the description of $MG_g$ as a cell complex we see that the terms in $dG$ are in one-to-one correspondence with the signed topological boundary of the cell $\Delta_G$. This is however not quite what we want. Firstly, we need to account for the fact that $G \sla e$ vanishes if $e$ is a tadpole. Secondly, we need to assign $G$ to a building block $\ti \Delta_G$ of $MG_g$; related to this is the fact that terms in $dG$ vanish if they have odd symmetries. 

The first problem is solved by considering the pair $(MG_g,MG_g^{\geq 1})$. 
Working modulo the subcomplex of graphs of positive weight is the geometric analog of dropping all terms from $dG$ whose loop number is less than $h_G$.

To solve the second problem we set up a sort of cellular chain complex generated by the family $\set{\Delta_G}_G$ that computes (relative) singular homology of $MG_g$. We sketch the construction from \cite{cgp} (see also \cite{hv,chkv}). Fix $g\geq 2$ and let
\[
C_k= \QQ \langle ( \Delta_G,\eta) \mid g(G)=g,e_G= k+1 \rangle / \sim 
\]
where $G=(G,w)$ is a genus $g$ weighted graph on $k+1$ edges, $\eta \in \det \QQ^{E_G}$ is an even orientation, and $\sim$ is the relation
\begin{align}
  & (G,w,-\eta)  = -(G,w,\eta),\tag{$i$} \\ 
  & (G,w,\eta) = (G',w', \varphi^* \eta) \text{ for any isomorphism } \varphi \colon (G',w) \overset{\sim}{\to} (G,w). \tag{$ii$}
\end{align}
Note that $\eta$ also specifies an orientation of both $\Delta_G$ and $\ti \Delta_G$ as subsets of $\PP^G=\PP(\RR^{E_G})$, hence the abuse of notation. To define the differential let 
\[
\pi\colon \QQ \langle ( \Delta_G,\eta) \rangle \longrightarrow \QQ \langle [\Delta_G,\eta]  \rangle = \QQ \langle ( \Delta_G,\eta)  \rangle / \sim
\]
 denote the quotient map. Then $\partial \colon C_k \to C_{k-1}$ is defined by 
\[
\partial [ \Delta_G, \eta ]  = \pi \circ \partial^{simp} (\Delta_G,\eta),
\]
where $\partial^{simp}$ is the usual simplicial boundary map.

To show that the complex $(C_\bullet, \partial)$ computes singular homology we construct a chain map $i$ as follows. Let $\beta(\Delta_G)$ denote the barycentric subdvision of $\Delta_G$. Composing $\beta$ with the quotient map $\Delta_G\to\ti \Delta_G$ induces a chain map 
\[
i \colon C_k \longrightarrow C_k^{sing}(MG_g). 
\]

\begin{prop} \label{prop:quasi_iso}
    The map $i$ is a quasi-isomorphism, $H(C_\bullet,\partial) \cong H^{sing}(MG_g)$.
\end{prop}

The proof uses the following 
\begin{lem}\label{lem:cone_over_boundary}
Each quotient cell $\ti \Delta_G$ is a cone over its boundary $\partial \ti \Delta_G$\footnote{More precisely, $\ti \Delta_G$ is a cone over the quotient of $\partial \Delta_G$; if $G$ has an odd symmetry, then the boundary of this cone has additional components. See \cref{fig:cubes_and_quotients} for an example.}. Therefore, 
\[
H_{e_G-1}^{sing}(\ti \Delta_G,\partial \ti \Delta_G) \cong \begin{cases} \QQ & \text{ if } [G]\neq 0, \\
0 & \text{ if } [G]=0.
\end{cases}
\] 
\end{lem}
\begin{proof}
    The cone statement is clear. The second statement follows from \cref{lem:quotient_homology}, since $\partial \ti \Delta_G$ is the quotient of the $(e_G-2)$-dimensional sphere $\partial \Delta_G$ by the group $\Aut(G)$, and $[G]=0$ if and only if $\Aut(G)$ contains an odd permutation.
\end{proof}

\begin{proof}[Proof of \cref{prop:quasi_iso}]
    Following \cite[Thm.2.27]{hatcher:AT} we filter $MG_g$ by subcomplexes $X^p$ consisting of graphs with at most $p+1$ edges (cells of dimension $\leq p$). For $C_\bullet$ we get an analogous filtration $\ldots \subset C_\bullet^p \subset C_\bullet^{p+1} \subset \ldots \subset C_\bullet$. 
    
    The map $i$ gives rise to a commutative diagram of long exact sequences
  \be \label{eq:les} \xymatrix{
      \cdots \ar[r] &  H_n(C_\bullet^p,C_\bullet^{p-1}) \ar[r] \ar[d] & H_{n-1}(C_\bullet^{p-1}) \ar[r] \ar[d] & H_{n-1}(C_\bullet^p) \ar[r] \ar[d] & \ldots \\ 
       \cdots \ar[r] &  H_n^{sing}(X^p,X^{p-1}) \ar[r] & H_{n-1}^{sing}(X^{p-1}) \ar[r]& H_{n-1}^{sing}(X^p) \ar[r] & \ldots
    }\ee
    
     By \cref{lem:cone_over_boundary} the groups $H_n^{sing}(X^p,X^{p-1})$ are trivial unless $n=p$ in which case 
    \[
    H_p^{sing}(X^p,X^{p-1})=\bigoplus_{G \colon e_G=p+1} H_{p}^{sing}(\ti \Delta_G,\partial \ti \Delta_G) \cong\bigoplus_{ \substack{ [G] \neq 0 \\  e_G=p+1}} \QQ.
    \]
    These groups are canonically isomorphic to 
    \[
    H_n(C_\bullet^p,C_\bullet^{p-1})=
    \begin{cases} 
        C_p & \text{ if } n=p,\\
        0 & \text{ else.}
    \end{cases}
    \]
    Therefore in \eqref{eq:les} all the vertical maps between relative homology groups are isomorphisms. It follows then by induction on $p$ and the five lemma that all vertical arrows are isomorphisms.
\end{proof}

An analogous statement holds for relative homology and cohomology of $MG_g$. 
Putting everything together and composing the map $i$ with 
\[
\comm_N^{(g)}  \longrightarrow C_\bullet,  \quad 
        [G,\eta]  \longmapsto [\Delta_G,\eta], 
\]
shows that 
\be \label{eq:iso_comm}
H_\bullet(\comm_N^{(g)}) \cong H_{\bullet+Ng-1}(MG_g,MG_g^{\geq 1})\cong\ti H_{\bullet+Ng-1}(MG_g),
\ee
where the last equality follows from the fact that the subcomplex $MG_g^{\geq 1}$ is contractible; see \cite{cgp}. 
\newline

The finite dimensionality of $MG_g$ (\cref{prop:dim_mod_space}) implies 
\begin{cor}
    $H_i(\comm_N^{(g)})=0$ unless $1-Ng\leq i \leq 3(g-1)-Ng$.
\end{cor}

\begin{rem} \label{rem:locfin}
Since $MG_g$ is compact, and $MG_g^{\geq1}$ is a subcomplex with complement $MG_g^{0}=MG_g \setminus MG_g^{\geq1}$, we can rephrase \cref{eq:iso_comm} using \textit{locally finite} (aka \textit{Borel-Moore}) homology,
\[ 
H_\bullet(\comm_N^{(g)}) \cong H_{\bullet+Ng-1}^{lf}(MG_g^0).
\]
\end{rem}

%%%%%%%%%%%%%%%%%%%%%%%%%%%%%%%%%%%%%

\subsection{$\gf_N$ for $N$ even ($\lie_N$ for $N$ odd)} \label{ss:lie_cellular}

A similar argument works for odd Lie graph homology. From the work of Culler and Vogtmann \cite{cv} we know that $MG_g^0$ and its spine $S_g$ (\cref{thm:spine_retract}) are rational classifying spaces for the group $\mathrm{Out}(F_g)$. We now show that the forested graph complex $\gf_N$, $N$ even, computes the homology of $S_g$.

Recall the cubical decomposition 
\[
S_g =  \Big( \bigcup_{(G,F)} \ti{c}(G,F)  \Big)/\sim .
\]
From this we proceed as in the previous section, replacing in each argument the cells $\Delta_G$ by cubes $c(G,F)$. The key ingredient for the compatibility of orientations is again a cone property of the quotient cubes $\ti{c}(G,F)$:

\begin{lem} \label{lem:cones_over_boundary_2}
Each quotient cube $\ti c(G,F)$ is a cone over its boundary $\partial \ti c(G,F)$. Therefore, 
\[
H_{e_F}(\ti c(G,F),\partial \ti c(G,F))\cong \begin{cases} \QQ & \text{ if $(G,F)$ has no odd symmetry,\footnotemark } \\
0 & \text{ else.}
\end{cases}
\]
\footnotetext{Note that $G$ alone is allowed to have odd symmetries; see \cref{sss:odd_symmetries}.} 
\end{lem}
\begin{proof}
See \cref{lem:cone_over_boundary}. 
\end{proof}

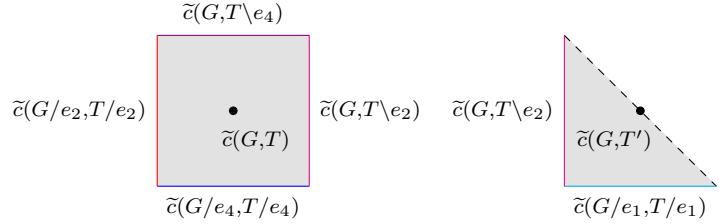
\begin{figure}[ht]
\begin{tikzpicture}
   \coordinate (ul) at (0,0);
   \coordinate  (ur) at (2,0);
   \coordinate (ol) at (0,2);
    \coordinate (or) at (2,2);
    \coordinate (c) at (1,1); 
   \filldraw[gray!23] (ul) -- (ur) --(or) -- (ol) -- (ul) node [xshift=1.3cm,yshift=.6cm]{$\color{black}  {\scriptstyle \ti c(G,T)}$};
   \draw[blue] (ul) -- (ur) node [midway,below]{$\color{black}  {\scriptstyle \ti c(G/e_4,T/e_4)}$}; 
   \draw[magenta] (ur) -- (or) node [midway,right]{$\color{black}  {\scriptstyle \ti c(G,T \setminus e_2)}$}; 
   \draw[violet] (or) -- (ol) node [midway,above]{$\color{black}  {\scriptstyle \ti c(G,T\setminus e_4)}$}; 
   \draw[red] (ol) -- (ul) node [midway,left]{$\color{black}  {\scriptstyle \ti c(G/e_2,T/e_2)}$}; 
      \filldraw[fill=black] (c) circle (0.05);
  \end{tikzpicture} 
    \begin{tikzpicture}
   \coordinate (ul) at (0,0);
   \coordinate  (ur) at (2,0);
   \coordinate (ol) at (0,2);
    \coordinate (or) at (2,2);
     \coordinate (c) at (1,1);
   \filldraw[gray!23] (ul) -- (ur) --(ol) -- (ul) node [xshift=.66cm,yshift=.6cm]{$\color{black}  {\scriptstyle \ti c(G,T')}$};
   \draw[dashed] (ol) -- (ur);
   \draw[cyan] (ul) -- (ur) node [midway,below]{$\color{black}  {\scriptstyle \ti c(G/e_1,T/e_1)}$}; 
   \draw[magenta] (ul) -- (ol) node [midway,left]{$\color{black}  {\scriptstyle \ti c(G,T \setminus e_2)}$}; 
   \filldraw[fill=black] (c) circle (0.05);
  \end{tikzpicture}
\caption{Two quotient cubes $\ti c(G,T)$ for the pairs $(G,\set{e_2,e_4})$ and $(G,\set{e_1,e_2})$ from \cref{fig:cubes_in_simplex}. Note that $(G,\set{e_1,e_2})$ has an odd symmetry. Both are cones over their boundary components, the black dots denoting the respective vertices of the cones.}
\label{fig:cubes_and_quotients}
\end{figure}

Combining everything as in \cref{ss:comm_cellular} we see that the forested graph complex computes the homology of $S_g$,
\[
H_\bullet(\gf_N)\cong H_{\bullet+Ng}(S_g)\cong H_{\bullet+Ng}(MG_g^0).
\]
Together with \cref{thm:Lie_and_Out} and \cref{cor:for_gc_and_lie_gc} we arrive at the desired connection between the moduli space of graphs and odd Lie graph homology,
\[
H_{\bullet} (\lie_N^{(g)}) \cong  \left( H_{(3-3N)g-3-\bullet}(MG_g^0) \right)^\vee .
\]

The dimension bounds on $MG_g^0$ or its spine $S_g$ (\cref{cor:dim_subspaces}) imply 
\begin{cor}
 For $N$ odd $ H_i(\lie_N^{(g)})=0$ unless $(1-3N)g\leq i \leq (3-3N)g-3$.
\end{cor}

\subsection{A cubical variant of $\comm_N$.}
The arguments in the previous section can also be used to define a cubical complex that computes commutative graph homology.

Recall the cubical decomposition \eqref{eq:MG_cube_decomposition} of $MG_g$: Each simplex $\Delta_G$ decomposes into an union of cubes $c(G,\gamma)$ where $\gamma \subsetneq G$. This decomposition descends to the quotients cells $\ti \Delta_G$ and $\ti c(G,\gamma)$. It follows that we can set up a complex of pairs $\gp_N$ which is generated by elements $(G,\gamma)$, $\gamma \subsetneq G$, without odd symmetries (i.e., there is no $\varphi \in \Aut(G)$ that induces an odd permutation on $E_\gamma$). The differential is the same as for $\gf_N$, $D=d-\delta$ where $d$ collapses edges in $\gamma$ and $\delta$ removes edges from $\gamma$ (\cref{eq:cube_diffential}).

    \subsection{Let's get odd} \label{ss:odd_cellular}
Also the graph complexes with opposite parity ($\comm_N$ for odd $N$, $\lie_N$ for even $N$) may be realized via chain complexes associated to moduli spaces in the manner described above. The previous case relied on the fact that orientations on $G$ are in one-to-one correspondence to orientations on the cells $\Delta_G$ (or $\ti{\Delta}_G$). For the opposite parity the same holds if we replace the trivial coefficient system $\QQ$ by a system of local coefficients. This was already mentioned by Kontsevich in \cite{ko1,ko2}. See also \cite[\S 5]{convogt} and \cite[\S 4]{Ww-Turchin:CommHairyGraphs}. 
\newline

%\subsubsection*{Digression on local coefficients} \label{sss:digression}
 To recall the definition of local coefficient systems (in the most basic case), we follow \cite{Steenrod:local_coeff} (see also \cite{hatcher:AT}).

A \textbf{bundle of groups} with fiber $\Gamma$ is a map of topological spaces $p\colon E \to X$ with: Every $x\in X$ has a neighbourhood $U$ such that there exists a homeomorphism $h\colon \inv{p}(U) \to U \times \Gamma$ ($\Gamma$ is endowed with the discrete topology) such that for any $y\in X$ the restriction $h|_{\inv{p}(y)} \colon \inv{p}(y) \to \set{y} \times \Gamma$ is a group isomorphism.

Let now $\Gamma$ be abelian and suppose that $X$ is a cell complex with the interior of every cell simply connected.

We choose in every cell $\sigma$ a point $x_\sigma \in \overset{\circ}{\sigma}$ and denote by $\Gamma_\sigma$ the group $\inv{p}(x_\sigma) \cong \Gamma$. If $\tau \subset \partial \sigma$, then using the simply connectedness of $\overset{\circ}{\sigma}$ and local trivializations of the bundle $p$ we get an isomorphism $h_{\sigma\tau} \colon \Gamma_\sigma \to \Gamma_\tau$, unique up to homotopy. 
\newline

Let $C(X;E)=\bigoplus_{n\geq 0} C_n(X;E)$ with
\[
C_n(X;E) = \left\{ \sum_i \ell_i\cdot \sigma_i \ \middle\vert \begin{array}{l}
    \sigma_i \colon \Delta^n \to X,  \ \ell_i \in \Gamma_{\sigma_i},\\
     \ell_i \cdot (-\sigma_i)= - \ell_i \cdot \sigma_i
  \end{array}\right\}
\]
and define $\partial \colon C_n(X;E) \to C_{n-1}(X;E)$ by
\[
\partial \sum_i \ell_i \cdot \sigma_i = \sum_i \sum_{j=0}^n (-1)^j h_{\sigma_i \sigma_i^j}(\ell_i) \cdot \sigma_i^j
\]
where $\sigma^j$ denotes the ``$j$-th face of $\sigma$". We have $\partial^2=0$, so $\left( C(X;E), \partial \right)$ is a chain complex. Its homology $H_\bullet(X;E)$ is called the \textbf{homology of $X$ with coefficients in $E$}.

\begin{example}
    If $E=X\times \Gamma$, then each $\Gamma_\sigma=\Gamma$, each $h_{\sigma \tau}=\id_\Gamma$, and therefore $H(X;E)$ is just the homology of $X$ with coefficients in $\Gamma$. 
\end{example}

\begin{example}
    Let  
    \[
    E=\bigcup_{ G\in \mathrm{ob}(\m G_g) } \ti \Delta_G \times \Gamma_G, \quad \Gamma_G = \QQ \otimes \det \homgra,
    \]
    with $p\colon E \to MG_g$ the projection onto the first factor. Then the homology of $MG_g$ with coefficients in $E$ computes odd commutative graph homology, 
   \[
   H_{n}(MG_g;E) \cong H_{n-Ng+1}(\comm_N^{(g)}) \quad (\text{for each odd } N.)
   \]
\end{example}

\begin{example}
    Let  
    \[
    E=\bigcup_{(G,F)\in \mathrm{ob} (\m P_g) } \ti c(G,F) \times \Gamma_G, \quad \Gamma_G =\QQ \otimes \det \homgra,
    \]
    with $p\colon E \to S_g$ the projection onto the first factor. Then the homology of $S_g$ with coefficients in $E$ computes even Lie graph homology, 
   \[
   \left( H_{n}(S_g;E) \right)^{\vee} \cong H_{(3-3N)g-3-\bullet}(\lie_N^{(g)}) \quad (\text{for each even } N.)
   \]
\end{example}

\begin{rem}
There are obviously more modern, more general and more effective ways to define and study homology with local coefficients. In particular cohomology with local coefficients can be formulated using a \textit{twisted de Rham complex} \cite{Aomoto:Th_of_Hg_fnctns} whose cochains are generated by ordinary differential forms, but the exterior differential gets an additional term $d\mapsto d + \omega \wedge$ arising from the wedge product with a 1-form $\omega$ (the \textit{twist}). In \cite{brown21} Francis Brown set up a de Rham theory for the moduli space of graphs to study even commutative graph homology. It would be interesting to see if the above notion of local coefficients can be translated into a twist $\omega_E$, and then use similar ideas to study odd commutative graph homology. 
\end{rem}

\subsection{Further simplification of the complexes} \label{ss:simplify_cplxs}

We use the geometric point of view to discuss different versions of our graph complexes.

\subsubsection{Graphs without bivalent vertices}
 In \cref{thm:bigger_cplxs} we claimed that $\comm_N^{\neq2}$ is quasi-isomorphic to $\comm_N$. With the established connection between graph complexes and moduli spaces of graphs we can show this geometrically. Furthermore, the result applies then automatically to the Lie case as well (and with minor adjustments also to the associative case).

Let $MG_g^{\neq2}$ denote the moduli space of weighted normalized metric graphs that may have bivalent vertices of weight zero, but must have at least one stable vertex.  

 \begin{prop} \label{prop:def_retract_bivalent_edges}
     $MG_g$ is a deformation retract of $MG_g^{\neq2}$. 
 \end{prop}

 \begin{proof}
     Consider a simplex $\Delta_G$ where $G$ is a graph with at least one bivalent weight zero vertex. Every such vertex lies on a unique edge-path between two stable vertices, or on a cycle based at one stable vertex. Suppose the edges in the path/cycle are labeled by $e_1,\ldots,e_k$. Let us denote a point in $\Delta_G$ by $[ x: y ]$ where $x=(x_1,\ldots,x_k)$ and $y$ represents the variables associated to the other edges of $G$, and by $[ x: y ]_\sim$ the corresponding equivalence class in $\ti \Delta_G$. Finally, let $G'=G/\set{e_2,\ldots,e_k}$ and note that all graphs $G/\set{e_1,\ldots,\widehat{e_i},\ldots,e_k}$ are isomorphic.

     Define a map $f_{G} \colon \ti \Delta_G  \times [0,1] \to \ti \Delta_G$ by
     \[
     \left( [x:y]_\sim, t \right) \longmapsto [ x_1 +  t\sum_{i=2}^k x_i : (1-t) x_2 : \ldots : (1-t) x_k :y   ]_\sim.
     \]
     It is continuous, $f_G(\cdot,0)$ is the identity on $\ti \Delta_G$, and $f_G(\cdot,1)$ satisfies $\image{f_G(\cdot,1)} = \ti \Delta_{G'}$ as well as ${f_G(\cdot,1)}|_{\ti \Delta_{G'}}=\id_{\ti \Delta_{G'}}$. This shows that $\ti \Delta_{G'}$ is a deformation retract of $\ti \Delta_G$ (if $G$ has additional bivalent vertices on different edge-paths/cycles, adjust the definition of $f_G$ in the obvious way).
     
     The restriction to a face $\set{y_j=0}$ agrees with the corresponding map $f_{G/e_j}$, so we can assemble the maps $f_G$ to define a deformation retraction of $MG_g^{\neq2}$ to $MG_g$.
 \end{proof}

We can repeat the construction in \cref{ss:comm_cellular} to show that $\comm_N^{\neq 2}$ computes relative homology of $MG_g^{\neq2}$. The argument is the same, although a little more convoluted since $MG_g^{\neq2}$ is infinite dimensional. Together with the previous proposition this implies then the desired result:

 \begin{cor}
     $H(\comm_N)=H(\comm_N^{\neq 2})$, and similar for $\lie_N$.
 \end{cor}
 
\subsubsection{Bridge-free graphs} Another simplification is to restrict the complexes and spaces to the case of bridge-free (aka core or 1-particle irreducible) graphs. For the moduli space of weight zero graphs $MG_g^0$ we can define a deformation retraction onto the subspace formed by bridge-free metric graphs by collapsing all bridges simultaneously, as in the proof of the previous proposition.

For the full moduli space $MG_g$ this does \emph{not} work, however. Algebraically phrased, $\comm_N$ is not quasi-isomorphic to the subcomplex spanned by bridge-free graphs (cf.\ \cref{rem:1vi}). To see this consider the example illustrated in \cref{fig:def_retract_handle}: The deformation retraction indicated there can be extended continuously to all of the red subcomplex except the vertex at the top, and the same is true for the quotient cell $\ti \Delta_G \subset MG_g$. 

\begin{figure}[h]
 \begin{tikzpicture}[scale=.9]
\coordinate (l) at (-3,0);
  \coordinate  (r) at (3,0); 
   \coordinate  (o) at (0,3);
   \draw[blue] (l) -- (r);
    \fill[blue!10] (l) -- (r) -- (o) -- (l);
   \draw[red!50] (l) -- (o) ;
   \draw[red!50] (r) -- (o) ;
 \fill[fill=red!50] (l) circle (0.06) node[below]{};
 \fill[fill=red!50] (r) circle (0.06) node[below]{};
 \fill[fill=red!50] (o) circle (0.06) node[above]{};
 %%% graphs
   \coordinate (v1) at (-.23,1.1);
  \coordinate (v2) at (.23,1.1);
   \draw[scale=2.5] (v1) to[in=-140,out=130,loop] node[midway,xshift=-.1cm] {} (v1);
   \draw (v1) -- node[above,yshift=-.1cm] {} (v2);
   \draw[scale=2.5] (v2) to[in=50,out=-40,loop] node[midway,xshift=.1cm] {} (v2);
  \filldraw[fill=black] (v1) circle (0.05);
  \filldraw[fill=black] (v2) circle (0.05);
  \coordinate (w1) at (0,-.5);
  \draw[scale=2.5] (w1) to[in=-140,out=130,loop] node[midway,xshift=-.1cm] {} (w1);
     \draw[scale=2.5] (w1) to[in=50,out=-40,loop] node[midway,xshift=.1cm] {} (w1);
  \filldraw[fill=black] (w1) circle (0.05); 
  \coordinate (r1) at (1.8,1.5);
  \coordinate (r2) at (2.3,1.5);
   \draw (r1) node[above,yshift=0cm] {$\scriptstyle{1}$} -- (r2);
   \draw[scale=2.5] (r2) to[in=50,out=-40,loop] node[midway,xshift=.1cm] {} (r2);
  \filldraw[fill=black] (r1) circle (0.05);
  \filldraw[fill=black] (r2) circle (0.05);
  \coordinate (l1) at (-1.8,1.5);
  \coordinate (l2) at (-2.3,1.5);
   \draw (l1) node[above,yshift=0cm] {$\scriptstyle{1}$} -- (l2);
   \draw[scale=2.5] (l2) to[in=-140,out=130,loop] node[midway,xshift=.1cm] {} (l2);
  \filldraw[fill=black] (l1) circle (0.05);
  \filldraw[fill=black] (l2) circle (0.05);
  \coordinate (ll) at (-3.5,0);
   \draw[scale=2.5] (ll) to[in=-140,out=130,loop] node[at end,xshift=.2cm] {$\scriptstyle{2}$} (ll);
  \filldraw[fill=black] (ll) circle (0.05);
  \coordinate (rr) at (3.5,0);
   \draw[scale=2.5] (rr) to[in=50,out=-40,loop] node[at end,xshift=-.2cm] {$\scriptstyle{2}$} (rr);
  \filldraw[fill=black] (rr) circle (0.05);
    \coordinate (lo) at (-.25,3.3);
  \coordinate (ro) at (.25,3.3);
  \draw (lo) node[left] {$\scriptstyle{1}$} -- (ro) node[right] {$\scriptstyle{1}$};
    \filldraw[fill=black] (lo) circle (0.05);
  \filldraw[fill=black] (ro) circle (0.05);
  \end{tikzpicture} 
   \caption{The simplex $\Delta_G$ for the dumbbell graph, depicted in the center. The red part belongs to the subcomplex of graphs with positive weight, its complement (in blue) deformation retracts onto the baseline (send the bridge's length to zero) which is an open cell in the weight zero subspace.} \label{fig:def_retract_handle}
\end{figure}
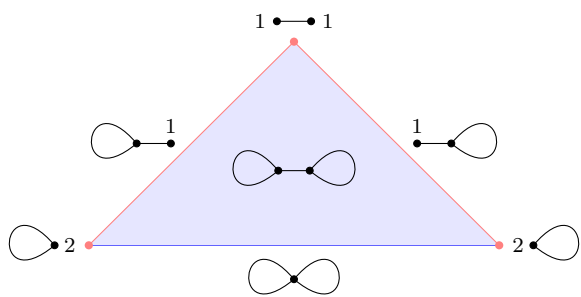

\subsubsection{Pairs $(G,F)$ where $G$ has on odd symmetry} \label{sss:odd_symmetries}
Let $N$ be even. Recall that in the forested graph complex $\gf_N$ only pairs $(G,F)$ without odd symmetries contribute. However, $G$ itself is allowed to have odd symmetries. Suppose $G$ is such a graph with odd symmetry. \Cref{lem:cone_over_boundary} implies that collapsing the cell $\ti \Delta_G$ to its cone point is a homotopy equivalence. In this way we could get rid of all cells indexed by graphs with odd symmetries, and obtain a space $MG_g^*$ which is homotopy equivalent to $MG_g$.

Since the forested graph complex arises from a cubical subdivison of $MG_g$, one might wonder why we cannot discard all pairs $(G,F)$, where $G$ has an odd symmetry, regardless of $F$, to get a smaller complex $\gf_N^*$ which is quasi-isomorphic to $\gf_N$.

The problem with this argument is that this new complex of pairs $\gf_N^*$ does not arise from a cubical subdivision of $MG_g^*$: This can be seen already for $g=2$ (which actually is a bad example, since $MG_2$ is contractible, but suffices to demonstrate the problem).

It would be interesting though to check whether the space $MG_g^*$ indeed allows for a ``nice" cellular decomposition which gives rise to a smaller (forested) graph complex.

\subsection{The associative graph complex} \label{ss:ass_cellular}

It is well-known that the homology of mapping class groups can be studied through a moduli space of ribbon graphs \cite{Stre84:QD,Penner:Fatgraphs}. In the same spirit as in \cref{s:moduli_space_of_graphs} one can define moduli spaces of ribbon graphs $MRG_g$, and then proceed as above to show that the associative graph complexes compute their homology groups; see \cite[Theorem 3.2]{ko1} or \cite{Igusa:GC_and_Ko_cycles}. 

Note that, if $G$ is a graph with ribbon structure $\sigma=\set{\sigma_v}_{v\in V(G)}$, then the automorphisms of $G$ that preserve $\sigma$ form a subgroup $\Aut(G,\sigma) \subset \Aut(G)$. As a consequence, $MRG_g$ is built from quotient cells $\Delta_G / \Aut(G,\sigma)$ that differ from the building blocks of $MG_g$. This means we can not realise the surface subcomplexes $\ass_N^{(g,S)}$ as subsets of $MG_g$---we have to use cells in $MRG_g$. There is however an interesting forgetful map between these spaces that discards the ribbon structure; see \cite[\S 3]{ko1}. 

To view both spaces from a unified point of view one has to go one step ``above", to Outer space, as explained in \cite{convogt} (see also \cite{Forrest:degree_thm_for_rgraphs}).

\bibliographystyle{alpha}
\bibliography{ref}

\end{document}